\documentclass[12pt,reqno]{amsart}

\usepackage[OT1]{fontenc}
\usepackage{amsthm,amsmath,amsfonts,amssymb,amsaddr}
\usepackage{natbib}
\usepackage{hyperref}

\usepackage{graphicx,here,comment}

\usepackage{fullpage}

\usepackage{times}
\usepackage{bm}
\usepackage{natbib}

\newtheorem{theorem}{Theorem}

\newtheorem{lemma}{Lemma}
\newtheorem{proposition}{Proposition}
\newtheorem{assumption}{Assumption}
\theoremstyle{definition}
\newtheorem{definition}{Definition}
\newtheorem{remark}{Remark}
\newtheorem{example}{Example}

\newcommand{\R}{\mathbb{R}}
\newcommand{\N}{\mathbb{N}}
\newcommand{\mF}{\mathcal{F}}

\newcommand{\mA}{\mathcal{A}}

\newcommand{\mI}{\mathcal{I}}

\newcommand{\mH}{\mathcal{H}}

\newcommand{\mT}{\mathcal{T}}

\newcommand{\mN}{\mathcal{N}}
\newcommand{\mU}{\mathcal{U}}
\newcommand{\mW}{\mathcal{W}}
\newcommand{\mX}{\mathcal{X}}

\newcommand{\mR}{\mathcal{R}}

\newcommand{\Ep}{E}
\renewcommand{\Pr}{\mathrm{Pr}}

\newcommand{\mE}{\mathcal{E}}

\renewcommand{\hat}{\widehat}
\renewcommand{\tilde}{\widetilde}

\newcommand{\argmin}{\operatornamewithlimits{argmin}}

\newcommand{\mone}{\textbf{1}}
\newcommand{\g}{\hat{f}_n}
\newcommand{\sgn}{\mathrm{sign}}

\DeclareMathOperator{\Cov}{Cov}

\title{Fast Convergence on Perfect Classification for Functional Data}

\author{Tomoya Wakayama$^\dagger$ \and Masaaki Imaizumi$^{\dagger \ddagger *}$}

\address{$^\dagger$The University of Tokyo, $^\ddagger$RIKEN Center for AIP}

\allowdisplaybreaks

\usepackage{color}

\begin{document}

\maketitle

\begin{abstract}
We investigate the availability of approaching perfect classification on functional data with finite samples. The seminal work (Delaigle and Hall (2012)) showed that perfect classification for functional data is easier to achieve than for finite-dimensional data. This result is based on their finding that a sufficient condition for the existence of a perfect classifier, named a Delaigle--Hall condition, is only available for functional data. However, there is a danger that a large sample size is required to achieve the perfect classification even though the Delaigle--Hall condition holds, because a minimax convergence rate of errors with functional data has a logarithm order in sample size. This study solves this complication by proving that the Delaigle--Hall condition also achieves fast convergence of the misclassification error in sample size, under the bounded entropy condition on functional data. We study a reproducing kernel Hilbert space-based classifier under the Delaigle--Hall condition, and show that a convergence rate of its misclassification error has an exponential order in sample size. Technically, our proof is based on (i) connecting the Delaigle--Hall condition and a margin of classifiers, and (ii) handling metric entropy of functional data. Our experiments support our result, and also illustrate that some other classifiers for functional data have a similar property.
\end{abstract}

\footnote{

Corresponding author. E-mail: imaizumi@g.ecc.u-tokyo.ac.jp.

$^\dagger$ The University of Tokyo: 7-3-1 Hongo, Bunkyo, Tokyo, 113-8654 JAPAN.

$\ddagger$ RIKEN Center for AIP: 1-4-1 Nihonbashi, Chuo, Tokyo, 103-0027 JAPAN.
}

\section{Introduction} 
The classification problem is one of the most general and significant problems in functional data analysis. 
A goal of this problem is, from functional data given in a form of (possibly) infinite-dimensional random curves, to predict labels or categories corresponding to the functional data. 
Because of its versatility, it has many applications such as science \citep{varughese2015non}, engineering \citep{gannaz2014classification,li2013hyperspectral,florindo2011enhancing}, medicine \citep{chang2014functional,islam2020classification,dai2017optimal}, and others. 
A variety of methods have been developed to solve this problem.
For example, distance-based methods \citep{alonso2012supervised, ferraty2003curves}, $k$-nearest neighbour methods \citep{biau2005functional,cerou2006nearest}, partially least square \citep{preda2007pls,preda2005clusterwise}, orthonormal basis methods \citep{delaigle2013classification,delaigle2012achieving}, Bayesian approaches \citep{wang2014efficient,yang2014bayesian}, and logistic regression methods \citep{araki2009functional}.
For a survey, for example, see \cite{cuevas2014partial}.

\textit{Perfect classification} for functional data was studied by \cite{delaigle2012achieving}, and it shows an advantage of using infinite-dimensional data.
This notion refers to the convergence of the misclassification error to zero under an optimal classifier, which is also referred to the realizability \citep{shalev2014understanding}.
Seminal works \citep{delaigle2012achieving, delaigle2013classification} showed that under certain conditions of mean and covariance functions of functional data (hereafter, we refer to the condition as a \textit{Delaigle--Hall condition}), there exists a classifier that achieves  perfect classification asymptotically.
This result does not usually hold with finite-dimensional data; hence, it explains the significance of dealing with infinite-dimensional vectors called functional data.
\cite{berrendero2018use} clarified a relation between the notion and a reproducing kernel.
Various methods have been shown to have a connection to a perfect classifier \citep{cerou2006nearest,dai2017optimal,cuesta2016perfect,hanneke2021universal}.

One difficulty is the possibility of a large sample size to achieve the perfect classification, suggested by a convergence analysis of misclassification error in sample size.
Nonparametric methods for functional classification are known to have a very slow convergence rate due to the infinite dimensionality of functional data.
Let $R(f)$ be a misclassification error under a classifier $f$. 
\cite{meister2016optimal} proved that any classifier $\tilde{f}_n$ consisting of $n$ observations has the following relationship with some data generating process:
\begin{align*}
    R(\tilde{f}_n) - \inf_{f}R(f) \geq c (\log n)^{-\alpha}.
\end{align*}
Here, $c > 0$ is an universal constant and $\alpha > 0$ is a parameter that depends on the data generating process.
This result shows the misclassification error of functional data cannot avoid errors that only decay on the logarithmic order in the general setting.
Since logarithmic decay is slower than every decay with a polynomial order, the convergence of this unavoidable error is very slow.
This implies that even if a perfect classifier exists, it can be difficult to benefit from it.

This study resolves the possibility mentioned above by showing that the Delaigle--Hall condition by \cite{delaigle2012achieving} also makes the convergence of the excess misclassification error sufficiently fast.
To the goal, we consider a reproducing kernel Hilbert space (RKHS) $\mH$ and study a classifier $\hat{f}_n \in \mH$ from $n$ observations by empirical loss minimization. 
Also, we consider a family of functional data which satisfies a bounded entropy condition, which implies the continuity and the boundedness of a norm of the functional data.

Then, we show that $\hat{f}_n$ obtains the following convergence under the Delaigle--Hall condition:
\begin{align*}
    \Ep \left(R(\Hat{f}_n) - \inf_{f \in \mH} R(f) \right) \leq 2\exp(-\beta n),
\end{align*}
with some parameter $\beta > 0$.
This exponential convergence in $n$ is faster than all polynomial convergence, which ensures that we can easily enjoy the perfect classification.

In the details of the classifier $\hat{f}_n$,
it is constructed as a linear sum of given kernel functions. 
Functional data analysis using RKHS is widely used in both linear and nonlinear regression problems \citep{preda2007regression,lian2007nonlinear,cai2012minimax,cui2020partially,tian2020additive}, while it has not been widely used for the classification problem with the exception of \cite{rincon2012wavelet}.
We remark that our approach is different from the study by \cite{berrendero2018use} that considers functional data as RKHSs since we construct the classifier using RKHSs.

As a technical contribution, our theoretical results are obtained by the following two ideas.
First, we introduce a \textit{hard-margin condition}, which describes the ease of classification problems, and connect it with the Delaigle--Hall condition.
In a general setting, a hard-margin condition is suitable for a perfectly classifiable setting such as \cite{koltchinskii2005exponential}.
We newly develop a hard-margin condition for functional data, then prove that the Delaigle--Hall condition implies the hard-margin condition by covariance structures of functional data. 
Second, we develop a metric entropy analysis on a classifier for functional data.
To analyze the speed of convergence of the empirical risk minimization classifier, we need to study an excess empirical risk on a space of classifiers.
However, since classifiers for functional data are more complicated than those of ordinary cases, we cannot use the traditional theoretical results.
We derived a new entropy bound for this purpose, which enabled us to develop the theory. To the best of our knowledge, both technical points are new theoretical results.

We note the bounded entropy condition on functional data for our result. 
First, the condition requires a kind of continuity of the functional data, e.g., the Lipschitz continuity.
Second, it requires that a norm of the functional data is bounded almost surely, which excludes, for example, Gaussian processes. 
These restrictions are necessary for our proof technique with an entropy condition.
To clarify this point, we provide several examples of stochastic processes that satisfy the entropy condition.

We conduct numerical experiments to confirm our theoretical findings.
It shows that the convergence speed of misclassification error by the RKHS method varies depending on whether the Delaigle--Hall and hard-margin conditions are satisfied or not. 
Further, we test additional several classification methods for functional data under the conditions.
Its result shows that not only the RKHS method but also non-parametric classification methods such as the Gaussian process method give similar effects on their convergence rates, but linear methods such as linear discriminant analysis do not cause such effect.

The remainder of the paper is organized as follows. 
Section \ref{sec:pre} introduces the setting and method that we focus on. Section \ref{sec:main} explains the perfect classification and its convergence result. In Section \ref{sec:experiment}, we confirm our theoretical result through experiments.
Section \ref{sec:conclusion} presents our conclusion and some discussion.
The supplementary material contains proofs and additional examples.

\subsection{Notation}
For $r \in \R$, $\sgn(r)$ is a sign function which is $1$ if $r > 0$, $-1$ if $r < 0$, and $0$ if $r=0$.
$\lceil r \rceil$ denotes the largest integer which is no more than $r$. 
For $r,r' \in \R$, $r\lor r' = \max\{r,r'\}$. 
For a function $f: \Omega \to \R$ on a set $\Omega$, $\|f\|_{L^\infty} = \sup_{x \in \Omega} |f(x)|$ denotes a sup-norm, and $\|f\|_n^2 = {n^{-1}\sum_{i=1}^n f(X_i)^2}$ is an empirical norm with observations $X_1,...,X_n$.
For an event $\mE$, $\mone\{\mE\}$ is an indicator function which is $1$ is $\mE$ holds, and $0$ otherwise.
For two sequences $\{a_n\}_{n \in \N}$ and $\{b_n\}_{n \in \N}$, $a_n \gtrsim b_n$ denotes that there exists a constant $c > 0$ such that $a_n \geq c b_n$ for all $n \geq \overline{n}$ with some finite $\overline{n} \in \N$. $a_n \lesssim b_n$ denotes its opposite. $a_n \asymp b_n$ means that both $a_n \gtrsim b_n$ and $a_n \lesssim b_n$ hold.
With a variable $z$, let $C_z$ be some positive and finite constant which only depends on $z$.
For a space $\Omega$ with a distance $d$ and $\delta > 0$, let $\mN(\delta,\Omega,d)$ be the covering number of $\mathcal{H}$, that is, the minimal number of balls that cover $\Omega$ with the radius $\delta$ in terms of $d$.

\section{Preliminary}\label{sec:pre}

\subsection{Problem Setting}
We consider a functional classification problem.
Let $\mX$ be a subset of an $L^2$-space on an index set $\mT \subset \R^d$ with some $d \geq 1$, and consider its inner product $\langle x, x' \rangle = \int_\mT x(t) x'(t) dt$ with $ x,x' \in \mX$ and its induced $L^2$-norm $\|\cdot\|$.
Let $\mathcal{B}(\mX)$ be an associated Borel $\sigma$-field of $\mX$.
Suppose we have observations $(X_1,Y_1),...,(X_n,Y_n)$ that are $n$ independent copies of random object $(X,Y)$ from a joint distribution $P$, where $X$ is a $\mX$-valued random function and $Y$ is $\{-1,1\}$-valued discrete random label. 
We write $w = P (Y=-1) \in (0,1)$.
Let $\Pi$ on $\mathcal{B}(\mX)$ be a marginal measure of $X$.
For each label, we define conditional measures on $\mathcal{B}(\mX)$ for functional data as $P_+ = \Pi( \cdot \mid Y=1)$ and $P_- = \Pi( \cdot \mid Y=-1)$.
By the definitions, we obtain $P_+(\mathcal{X})=P_-(\mathcal{X})=1$ and $\Pi=w P_-+(1-w)P_+$.

The goal of this problem is to construct a classifier that outputs a label from a functional input in $\mX$.
For a given function $f:\mathcal{X}\rightarrow \R$, a corresponding binary classifier is defined as $ \sgn\circ f$.
We also define a misclassification error by $f$ as
\begin{align*}
    R(f)=P\{(X,Y):Y\neq \mathrm{sign}(f(X))\},
\end{align*}
which is also referred as a generalization error of the classification problem.

We discuss the existence of a minimizer of $R(f)$, which is referred to as an optimal function for the Bayes classifier.
To this end, we have to develop a density function of $P_+$ and $P_-$.
Unlike the finite-dimensional data case, it is not trivial to define the densities, since function spaces do not have the useful Lebesgue measure.
Instead, we utilize $\Pi$ as a base measure, which is absolute continuous to $P_+$ and $P_-$, and define the following densities by the Radon--Nikodym derivative $p_+ = dP_+ / d \Pi$ and $p_- = dP_- / d \Pi$.
The following result shows that we can guarantee the function as a minimizer with $p_+$ and $p_-$.
Its proof is deferred to the supplementary material.
\begin{lemma} \label{lem:classifier}
We define a function $f_0: \mX \to \R$ as  
    $f_0(x)=(1-w) p_+(x) - w p_-(x)$.
Then, $f_0$ minimizes $R(f)$.
\end{lemma}

\subsection{Methodology: RKHS Classifier}

We provide a setting about the notion of a reproducing kernel Hilbert space (RKHS) for functional data.
Let $\mathcal{H}$ be a Hilbert space on $\mX$.
Also, let $\langle \cdot, \cdot \rangle_\mH$ be an inner product of $\mH$, and $\|\cdot\|_\mH$ be an induced norm of $\mH$.
A function $K: \mX \times \mX \to \R$ is referred to be a reproducing kernel for $\mH$, if it satisfies (i) for every $x \in \mX$, $K(\cdot, x) \in \mH$ holds, and (ii) for every $x \in \mX$ and $f \in \mH$, $f(x) = \langle f, K(\cdot, x) \rangle_\mH$ holds.
It is well-known that a reproducing kernel $K$ is symmetric, nonnegative definite, uniquely determined by an RKHS $\mH$.
Also, a set of linear form $\{\sum_{i=1}^n c_i K(x_i,\cdot): c_i \in \R, x_i \in \mX\}$ is dense in $\mH$.
For details, see \cite{berlinet2011reproducing}.

As a property of RKHS that is important to our study, for any $x,x' \in \mX$ and $f \in \mH$, there exists a constant $c_\mH$ such that
\begin{align}
    &|f(x)| \leq c_\mH\|f\|_\mH, \mbox{~and~}|f(x)-f(x')|\leq \|f\|_\mH \|x - x'\| \label{ineq:prop_rkhs}
\end{align}
holds.
For its proof, see Proposition 4.30 in \cite{steinwart2008support}.
Hereafter, we set $c_\mH = 1$ without loss of generality.
Further, we impose that $\mH$ is dense in a set of continuous functions $C(\mX)$.
This property is referred to \textit{universality} and many common RKHSs are known to satisfy it (see Definition 4.52 and Corollary 4.55 in \cite{steinwart2008support}).

In the binary classification problems, we define a classifier $\hat{f}_n \in \mH$.
Let $\ell: \R \to \R$ be a loss function such that $\ell\geq \mone_{(-\infty,0]}$ is decreasing, bounded by $1$, convex, and $1$-Lipschitz continuous. 
We consider the following optimization problem:
\begin{align}
    \hat{f}_n=\argmin_{f\in\mathcal{H}} \frac{1}{n}\sum_{i=1}^n \ell(Y_if(X_i))+\lambda \|f\|_\mH^2\,, \label{def:erm}
\end{align}
where $\lambda > 0$ is a regularization coefficient.
In practice, this minimization problem is solved in various ways, depending on the loss function. 
We further assume that $\ell$ is twice differentiable, its first derivative $\ell'$ is negative, increasing and bounded below by $-1$, and its second derivative $\ell''$ is bounded above by $1$.
A logit loss $\ell(t) = \log (1 + \exp(-t))$ is a common choice as an example of a loss function that satisfies the requirement.

\section{Main Result on Convergence of Misclassification Error}\label{sec:main}

Our aim is to study an excess risk $R(\hat{f}_n)- \inf_{f \in \mH}R(f)$ under a perfect classifiable setting.
To this goal, a certain assumption for perfect classification plays a key role.

\subsection{Delaigle--Hall Condition for Perfect Classification}

The Delaigle--Hall condition, established by \cite{delaigle2012achieving,delaigle2013classification}, is a condition for functional data to be asymptotically perfect classifiable.

We provide some notations.
Consider covariance functions of $X \sim \Pi$ as $C(t,t') = \mathrm{cov} (X(t),X(t'))$, which is assumed to be strictly positive definite and uniformly bounded.
Further, a random function $X_+$ and $X_-$ drawn from $P_+$ and $P_-$ with corresponding bounded covariance function respectively $C_{\ell}(t,t') = \mathrm{cov}(X_\ell(t),X_\ell(t'))$ for $\ell \in \{-,+\}$. 
Their spectral decompositions (e.g. Theorem 4.6.5 in \cite{hsing2015theoretical}) are written as
\begin{align*}
    C(t,t') = \sum_{j=1}^\infty \theta_{j} \phi_j(t)\phi_j(t'), \mbox{~and~}C_{\ell}(u,v) &=\sum_{j=1}^{\infty}\theta_{\ell j}\phi_{\ell j}(u)\phi_{\ell j}(v),
\end{align*}
where $\{\theta_{\ell j},\phi_{\ell j}\}_{j=1}^\infty$, and $\{\theta_j,\phi_j\}_{j=1}^\infty$ are pairs of non-zero eigenvalues and eigenfunctions of $C_{\ell}$ and $C$ for $\ell \in \{-,+\}$, respectively.
We assume that they are sorted as $\theta_1 \geq \theta_2 \geq \ldots$ and $\theta_{\ell,1} \geq \theta_{\ell,2} \geq \ldots$.
We also introduce coefficients of mean functions for each label.
Let $X_+$ and $X_-$ be random functions generated from $P_+$ and $P_-$ respectively, then we define its mean as
\begin{align*}
\mu_+ := \Ep_{P_+}[X_+] = \sum_{j=1}^\infty \mu_{+,j} \phi_j, \mbox{~~and~~}\mu_- := \Ep_{P_-}[X_-] = \sum_{j=1}^\infty \mu_{-,j} \phi_j,
\end{align*}
with coefficients $\mu_{+,j}$ and $\mu_{-,j}$ for $j \in \N$ by the generalized Fourier decomposition.
Using the basis $\{\phi_j\}_{j=1}^{\infty}$, we express there difference $\mu_+-\mu_-=\sum_{j=1}^{\infty}\mu_j\phi_j$ by coefficients $\mu_j$ for $ j \in \N$.

We introduce a condition for perfect classification with the notation.
The following condition is developed in section 4.2 in \cite{delaigle2012achieving}:
\begin{definition}[Delaigle--Hall condition \citep{delaigle2012achieving}]
    The joint measure $P$ satisfies the Delaigle--Hall conditions, if the following holds for $\ell \in \{+,-\}$:
\begin{align}
    \lim_{M\rightarrow\infty}\frac{(\sum_{j=1}^M\theta_j^{-1}\mu_j^2)^2}{\sum_{j=1}^{\infty}\theta_{\ell j} (\sum_{i=1}^M \theta_i^{-1}\mu_i \int \phi_i(u)\phi_{\ell j}(u) du)^2}=\infty. \label{cond:DH}
\end{align}
\end{definition}
Further, we can simplify the condition \eqref{cond:DH} under a specific setting: if $C_+$ and $C_-$ has common eigenfunctions, that is, $\phi_j = \phi_{+j} = \phi_{-j}$, the condition \eqref{cond:DH} is rewritten to a more intuitive formulation
\begin{align*}
        \sum_{j=1}^\infty \theta_{j}^{-1} \mu_{j}^2 = \infty. 
\end{align*}
This condition indicates that the covariance of functional data is too concise compared to the mean difference.
If each functional data are nearly independent for each input, the Delaigle--Hall condition is more likely to be satisfied because $\theta_j$ decays faster as $j$ increases.
The Delaigle--Hall conditions implies the following result:
\begin{proposition}[Theorem 1 in \cite{delaigle2012achieving}] \label{prop:DH}
    If the Delaigle--Hall condition is satisfied and $X\mid Y$ is Gaussian, then there is a perfect classification, that is, $\inf_f R(f) = 0$ holds.
\end{proposition}
\noindent
A similar result also holds without the Gaussianity of $X|Y$ (see Theorem 2 in \cite{delaigle2012achieving}).

This Delaigle--Hall condition gives a sufficient condition for Gaussian measures on infinite dimensional spaces to be mutually singular, based on the classical H\'ajek-Feldman theorem \cite{da2014stochastic}.
Since this type of singularity appears more easily than on finite-dimensional spaces, this shows one advantage of using infinite-dimensional functional data.
We provide an example of functional data distributions satisfying the Delaigle--Hall condition.

\begin{example}[Decaying Coefficients]
    We give a specific example of $\theta_j$ and $\mu_j$, and consider when the Delaigle--Hall condition is satisfied.
    Suppose that $\theta_j \asymp j^{-\alpha}$ with $\alpha > 0$ and  $\mu_j \asymp j^{-\beta}$ with $\beta > 0$ hold.
    Then, the Delaigle--Hall condition is satisfied with $2\beta - \alpha \leq 1$.
    Since $\alpha$ describes a complexity of the covariance $C(t,t')$ and $\beta$ is for smoothness of $\mu$, the Delaigle--Hall condition is more likely to be satisfied when the functional data is less smooth and the covariance decays quickly.
\end{example}

\subsection{Conditions}
We introduce assumptions for the fast convergence. 
Recall that $\mN(\varepsilon, \mX, d)$ denotes a covering number of $\mX$, which is common for the empirical process theory and statistical learning theory (for introduction, see \cite{van1996weak}).
We consider the following condition:
\begin{assumption}[Covering Bound]\label{asmp:covering}
    There exists constants $\bar{\varepsilon}>0$, $\gamma>0$, and $V>0$ such that for every $\varepsilon \in (0, \bar{\varepsilon})$, the following holds:
\begin{align*}
    \log \mN(\varepsilon,\mathcal{X},d)\leq V \varepsilon^{-\gamma}.
\end{align*}
\end{assumption}
\noindent
This type of assumption has been used by \cite{meister2016optimal} for convergence analysis of functional data.
$\gamma$ represents the complexity of the functional data space $\mX$ and controls the smoothness of $\mX$ and a dimension of inputs for $X$.

Assumption \ref{asmp:covering} restricts a form of functional data in $\mX$ in the following ways.
First, it requires a kind of continuity or differentiability for the functional data. 
The degree of smoothness is adjusted by the decay rate $\gamma$ in Assumption \ref{asmp:covering}.
Second, it also requires that a norm of the functional data is bounded.
This constraint excludes Gaussian processes, hence we have to use a truncated version of Gaussian processes instead.
We present examples that satisfy Assumption \ref{asmp:covering}.
Some of them are deferred to the supplementary material.

\begin{example}[Smooth Path] \label{example:smooth}
    For $\alpha  \in \N$, suppose that, $\mX$ is a set of functions $f$ on $[0,1]^p$ which has partial derivatives up to an order $\alpha -1$ that are uniformly bounded by some constant, and the highest partial derivatives are Lipschitz continuous.
    In this case, a setting $\gamma = p / \alpha$ satisfies Assumption \ref{asmp:covering} with $d = \|\cdot\|_{L^\infty}$ (Theorem 2.7.1 in \cite{van1996weak}). 
    \qed
\end{example}

\begin{example}[Non-Smooth Path] \label{example:holder}
    For $\alpha' \in (0,1]$, we consider a case that $\mX$ is a set of $\alpha'$-H\"older-continuous functions on $[0,1]^p$, which is a set of functions $f:[0,1]^p \to \R$ such that
    \begin{align*}
        |f(x) - f(x')| \leq C \|x-x'\|^\alpha,
    \end{align*}
    holds for every $x,x' \in [0,1]^p$ with some constant $C>0$.
    In this case, a setting $\gamma = p / \alpha$ satisfies Assumption \ref{asmp:covering} with $d = \|\cdot\|_{L^\infty}$ (Theorem 2.7.1 in \cite{van1996weak}).
    Note that this set includes non-differentiable functions.
    \qed
\end{example}

\begin{example}[Unbounded Path with Finite Peaks] \label{ex:unbouded}
We consider a family of functions $f$ on $[0,1]$ such that
\begin{align*}
    f(x) = g(x) + \sum_{j=1}^J \psi(x; a_j,t_j),
\end{align*}
where $g$ is a function from the Sobolev space with an order $\alpha \in N$ (a space of $\alpha$-times weakly differentiable functions in terms of $\|\cdot\|_{L^2}$), $\psi(x;a_j,t_j) = a_j/|(x-t_j)|^{1/3}$ is an unbounded peak function with scale parameters $a_j \in [0,1]$ and fixed locations $t_j \in [0,1]$, and $J$ is a number of peaks.
We can show that a set of such the functions satisfies Assumption \ref{asmp:covering} is satisfied with $\gamma = 1/\alpha$ and with $d = \|\cdot\|_{L^2}$ and sufficiently large $V>0$.
Specifically, see Proposition \ref{prop:unbounded} in the supplementary material.
\qed
\end{example}

We give the second condition for the distribution $\Pi$ of $X$.
For $x \in \mX$ and $\delta > 0$, define $B(x;{\delta})$ as an $x$ centered open ball with radius $\delta$ in terms of $\|\cdot\|$. Then, we impose the following assumption.
\begin{assumption}[Positive Small Ball Probability]\label{asmp:smallball} For any $x$ in a support of $\Pi$ and $\delta > 0$, 
$\Pi (B(x;{\delta})) > 0$ holds.
\end{assumption}
\noindent
This assumption is satisfied a general class of distributions even in the functional data setting.
We provide several examples in the supplementary material.

\subsection{Convergence Result}
We provide our main result on the convergence speed of the generalization error.
We provide a proof outline in the next section and the full proof in the supplementary material.
Details of constants will be explained in the same section.

\begin{theorem} \label{thm:convergence}
Let $\mathcal{H}$ be an RKHS on $\mX$ with a universal kernel.
Let $\hat{f}_n \in \mH$ be a classifier, minimizing the empirical loss as defined in \eqref{def:erm}.
Suppose that the Delaigle--Hall condition and Assumption \ref{asmp:covering} and \ref{asmp:smallball} hold. 
Then, there exist positive and finite constants $C_{V,\gamma} $ and $ C_{V,\Pi,\mH}$, such that the following inequality holds with any $\lambda \in [ \underline{\lambda}, \overline{\lambda}]$ with $\underline{\lambda} =  \max\{ (\log n)^{-{1} / {\gamma}} , C_{V,\gamma} \log \log n /n \}$ and $\overline{\lambda} = C_{V,\Pi,\mH}$, and any $n \in \N$ as $\overline{\lambda} \geq \underline{\lambda}$: 
\begin{align*}
    \Ep\left[R(\hat{f}_n) - \inf_{f \in \mH} R(f)\right]\leq 2\exp(-\beta n),
\end{align*}
where $\beta>0$ is a parameter which depends on $\mathcal{H}, \Pi$ and $V$.
\end{theorem}
\noindent
The result shows that the very fast convergence of the generalization error is obtained under the Delaigle--Hall condition and sufficient sample size.
In other words, since this convergence is exponential in $n$, we can see that the error decays faster than all polynomial convergence in $n$.
This is contrastive to the logarithmic convergence by \cite{meister2016optimal}, that is, $R(\tilde{f}_n) - \inf_{f}R(f) \geq C (\log n)^{-1/\gamma}$ holds.
This suggests that adding the Delaigle--Hall condition reduces the complexity of the functional classification problem more than expected.

Technically, the following two points are important.
First, the minimum number of required $n$ is determined by $\gamma$, which reflects the complexity of functional data in Assumption \ref{asmp:covering}.
When functional data are more complex, that is $\gamma$ is large, the required sample size increases.
Second, $\beta$ also depends on various parameters and it is complicated to describe.
Its rigorous values will be provided in the full proof.

\begin{remark}[Role of RKHS]
We utilize RKHSs for classifiers for the following reasons.
First, the property of pointwise bound \eqref{ineq:prop_rkhs} by RKHSs is important for the error analysis.
Second, an RKHS is closely related to the Delaigle--Hall condition \eqref{cond:DH}, since the condition is regarded as measuring the difference of means of the distributions in terms of an RKHS norm.
This relation makes our error analysis simple.
\end{remark}
\begin{remark}[Selection of $\mH$]
We discuss the effect of the choice of RKHS. 
The exponential convergence in $n$, which is the main claim of Theorem \ref{thm:convergence}, holds for all RKHSs regardless of their choice as long as the requirements are satisfied. 
\end{remark}

\subsection{Proof Overview}

The proof of Theorem \ref{thm:convergence} contains the following three steps: (i) rewrite the Delaigle--Hall condition by a hard-margin condition, (ii) decomposition of the misclassification error, and (iii) study each of the components. 
Hereafter, we set $L$ as a Lipschitz constant of $f^*$ and assume that $\|f^*\|_\mH \geq 1$ without loss of generality.

\textbf{Step (i): Rewrite the Delaigle--Hall condition to the hard-margin condition}:
We first introduce the hard-margin condition, which is a general condition for various classification problems:
\begin{definition}[Hard-Margin Condition] \label{asmp:hard-margin}
    A margin of $\Pi$ with $f: \mX \to \R$ is defined as
    \begin{align*}
     \delta(f, \Pi) =\sup\left\{\delta: \Pi(\{x:|f(x)|<\delta \})=0\right\},
    \end{align*}
    We say $\Pi$ satisfies the hard-margin condition with given $f$, if $\delta(f,\Pi) > 0$ holds.
\end{definition}
This condition requires that a discrepancy between sets $\{x:f(x)>\delta\}$ and $\{x:f(x)<-\delta\}$ is large almost surely.
In other words, the margin with $f$ is contained in a $\Pi$-null set.
A margin is a useful notion to handle the difficulty of classification problems.
This condition is related to a common condition for other classification problems, referred to as a strong noise condition \citep{koltchinskii2005exponential,audibert2007fast}.

To show connection between the Delaigle--Hall and the hard-margin condition, we introduce $f^*$ as follows.
We define a sum of orthogonal basis $\psi_M = \sum_{j=1}^M \theta_j^{-1} (\mu_{+,j} - \mu_{-,j}) \phi_j$ and $f^*_M$ as
\begin{align*}
    f^*_M(x) =  \left( \left\langle x  -  \mu_+, \psi_M\right\rangle \right)^2 - \left( \left\langle x  - \mu_-, \psi_M  \right\rangle \right)^2.
\end{align*}
Also, we define $f^* = \lim_{M \to \infty} f_M^*$.
This function measures whether the input $x$ is closer to $\mu_+$ or $\mu_-$ with the weight $\psi_\infty$, and a sign of $f^*(x)$ works as a classifier.
Then the following result shows the equivalence property between the Delaigle--Hall and hard-margin conditions:
\begin{proposition}[Delaigle--Hall implies hard-margin]\label{prop:dhhm}
If the Delaigle--Hall conditions holds, then, $\delta(f^*, \Pi) > 0$ holds.
\end{proposition}
\noindent
Proposition \ref{prop:dhhm} shows that the Delaigle--Hall condition leads to the margin of $\Pi$ being positive, and this implication is similar to Theorem 5 in \cite{berrendero2018use}, which states that Delaigle--Hall condition is equivalent to a discrepancy of supports of two measures $P_+$ and $P_-$ under the Gaussian homoscedastic model.
Also, Proposition \ref{prop:dhhm} shows that $f^*$ is an effective classifier with a sufficiently large margin under the Delaigle--Hall condition. 
This proof is based on an idea in \cite{delaigle2012achieving}, which applies a property that distances between functional data become infinitely large under the weighting by $\theta_j$ from the covariance.

\textbf{Step (ii): Generalization Error Decomposition}:
In preparation, we convert the perfect classifier $f^*$ into a controllable form.
To the aim, we define $\tilde{f}_M(x) := f_M(x)/|f_M(x)|$ and $\tilde{f}^*(x)= \lim_{M\to\infty}\tilde{f}_M(x)$. Since the risk only depends on the sign of $f^*$, we have $R(\tilde{f}^*)=R(f^*)$.
In the following, we study the classification error based on $\tilde{f}^*$ rather than $f^*$.

The first step is to rewrite the generalization error to an integral which involves a probability associated with the signs of $\tilde{f}^*$ and $\hat{f}$. 
The standard calculation yields the following transformation by the Bochner integral:
\begin{eqnarray*}
    \Ep[R(\hat{f}_n)-R(\tilde{f}^*)]\leq \int |\eta(x)| \, \Pr(\hat{f}_n(x) \tilde{f}^*(x)\leq 0) \,d\Pi(x),
\end{eqnarray*}
where $\eta(x) = \Ep[Y|X=x]$.
Next, for each $x$, we decompose the probability term $ \Pr(\hat{f}_n(x) \tilde{f}^*(x)\leq 0)$.
For $x$ such that $\tilde{f}^*(x) > 0$ holds, the misclassification error is rewritten as
\begin{align*}
    \Pr(\hat{f}_n(x) \tilde{f}^*(x)\leq 0)=\Pr(\hat{f}_n(x)\leq 0)\leq \underbrace{\Pr(\hat{f}_n(x)\leq 0\,,\,\|\hat{f}_n\|_\mH\leq U)}_{= T_1}+ \underbrace{\Pr(\|\hat{f}_n\|_\mH>U)}_{= T_2}, 
\end{align*}
with a threshold value $U > 0$ which will be specified in the full proof.
We divide the event by the value of $\|\hat{f}_n\|_\mH$ associated with $U$, then study each of the probability terms separately.

\textbf{Step (iii): Bound the Probability Terms:} For $T_1$, we bound it using the hard-margin condition.
Let us define $L_n(f) := n^{-1} \sum_{i=1}^n \ell(Y_i f(X_i)) + \lambda \|f\|_{\mH}^2$.
We show that $\hat{f}_n$ cannot be a minimizer of $L_n(f)$ as \eqref{def:erm}, when $T_1$ is large under the hard-margin condition.
Then, by the contradiction, we prove that $T_1$ converges exponentially in $n$.
This part mainly follows the same proof in \cite{koltchinskii2005exponential}.

For $T_2$, we bound it by the empirical process technique.
This part is specific to functional data, hence some theory such as \cite{koltchinskii2005exponential} does not work.
First, show that to bound an excess loss $L_n(\hat{f}_n) - L_n(\tilde{f}^*)$ is sufficient to achieve the goal.
To show the convergence of the excess loss, we develop a covering number bound for $\mH$ (Lemma \ref{lem:covering} in the supplementary material) and develop the following bound with probability at least $1-\exp({-t})$:
\begin{align*}
    |L_n(f) - L(f)|\leq Rc_{V,\gamma}{(\log n)^{-1/\gamma}}+\sqrt{{2t}/{n}},
\end{align*}
for any $f \in \mH$ such that $\|f\|_\mH \leq \|f^\dagger\|_\mH$ holds and any $t > 0$ (Lemma \ref{lem:empirical_bound} in the supplementary material).
Here, $c_{V,\gamma}$ is a constant depending on $V$ and $\gamma$, which will be specified in the full proof.
As a consequence, a sufficiently large $n$ achieves the goal.

\section{Experiments}\label{sec:experiment}

We conduct numerical experiments to support this theoretical result, that is, we analyze the change in the convergence rate of various classification methods for functional data under the Delaigle--Hall and hard-margin conditions.

\subsection{Experimental Setting}

For the functional classification problem, we consider the following settings.
We generate functional data from two groups with labels $\{-1,1\}$. For each group, we generate $n$ functions on $\mT = [0,1]$ with a orthogonal basis $\phi_0(t) = 1$ and $\phi_j(t) = \sqrt{2} \sin (\pi jt), \forall j \geq 1$.
$n$ is set from $1$ to $3000$.
For a label $+1$, we generate functional data $X_{i+}(t) = \sum_{j=0}^{50} (\theta_j^{1/2} Z_{j+} + \mu_{j+}) \phi_j(t)$ with random variables $Z_{j+}$ and coefficients $ \theta_j, \mu_{j+}$ for $j=0,1,...,50$ and $i=1,...,n$.
Similarly, for a label $-1$, we generate $X_{i-}(t) = \sum_{j=0}^{50} (\theta_j^{1/2} Z_{j-} + \mu_{j-}) \phi_j(t)$ with random variables $Z_{j-}$ and coefficients $ \mu_{j-}$.

We consider the following two scenarios, and the values of the random variables and coefficients are determined separately.
In \textit{Scenario 1}, to consider perfect classifiable data by the Delaigle--Hall condition, we set $\theta_j = j^{-2}$, $\mu_{j-}=0$, and change $\mu_{j+} = j^{-\gamma}$ and draw $Z_{j+},Z_{j-}$ from standard normal Gaussian.
Here, $\gamma$ handles the complexity of a mean of functional data and thus determines the data generating process satisfies/violates the Delaigle--Hall condition. 
If $\gamma \leq 3/2$, the Delaigle--Hall condition is satisfied or violated otherwise. 
In \textit{Scenario 2}, we examine perfect classification according to the hard-margin condition. 
We set $\theta_j = j^{-2},$ $\mu_{j-}= 0$ and adjust $\mu_{j+} = \mone\{j = 0\} \mu$,  and let $Z_{j+},Z_{j-}$ be from uniform distribution on $[-1/2,1/2]$.
Here, $\mu$ is a key parameter to satisfy/violate the hard-margin condition. 
If $\mu \geq 1$ holds, the hard-margin conditions are satisfied, since domains of $P_+$ and $P_-$ do not overlap with each other.
Otherwise, the hard-margin condition is violated. 
With each method and $n$, we study its misclassification rate with $1000$ newly generated data for test. 
We repeat each simulation experiment $200$ times and report its mean. 
The case where basis functions differ between labels is discussed in the supplemental material.

\subsection{RKHS Classifier and the Delaigle--Hall/hard-margin Condition}

We study the misclassification rate of the RKHS method in \eqref{def:erm}.
We set the loss function as the logit loss $\ell(u) = \log (1+ \exp({-u}))$, and the hypothesis space $\mH$ is constructed by the functional RKHS associated with the Gaussian kernel $k(x,x') = \exp(-\|x-x'\|^2/h)$ with functions $x = x(t)$and $ x'=x'(t)$, and a hyper-parameter $h > 0$.
The norm in the kernel is calculated as $\sum_{j=0}^\infty (\xi_j - \xi_j')^2$ where $\xi_j = \langle x, \phi_j \rangle$ and $\xi_j' = \langle x', \phi_j \rangle$. 
By the representer theorem (Theorem 5.5 in \cite{steinwart2008support}), the minimization problem is rewritten as
\begin{align*}
    \min_{\{w_j\}_{j=1}^n} \frac{1}{n} \sum_{i=1}^n \ell\left( Y_i \sum_{j=1}^n w_j k(X_i,X_j) \right) + \lambda \sum_{j=1}^n w_j^2,
\end{align*}
with the parameters $w_1,...,w_n$.
We solve the optimization problem by the gradient descent method. The bandwidth $h$ and the penalized parameter $\lambda$ are determined by cross-validation (CV) from $\{2^{-5},2^{-4},\ldots,2^{4}\}$, minimizing the misclassification rate with newly generated test data.

In Scenario 1, we consider configurations of the mean decay parameter as $\gamma \in \{1.6,1.7\}$ to satisfy the Delaigle--Hall condition, or $\gamma \in \{ 1.3,1.4\}$ to violate the condition.
In Scenario 2, we consider $\mu \in \{0.8,0.9\}$ to satisfy the hard-margin condition, or $\mu \in \{1.1,1.2\}$ to violate it.
For each scenario, we plot \textit{error} (logarithm of misclassification error) against $\log n$ in Figure \ref{fig:proposed}. 

\begin{figure}[htbp]
\centering
  \hspace*{-1cm}
  \includegraphics[width=0.9\hsize]{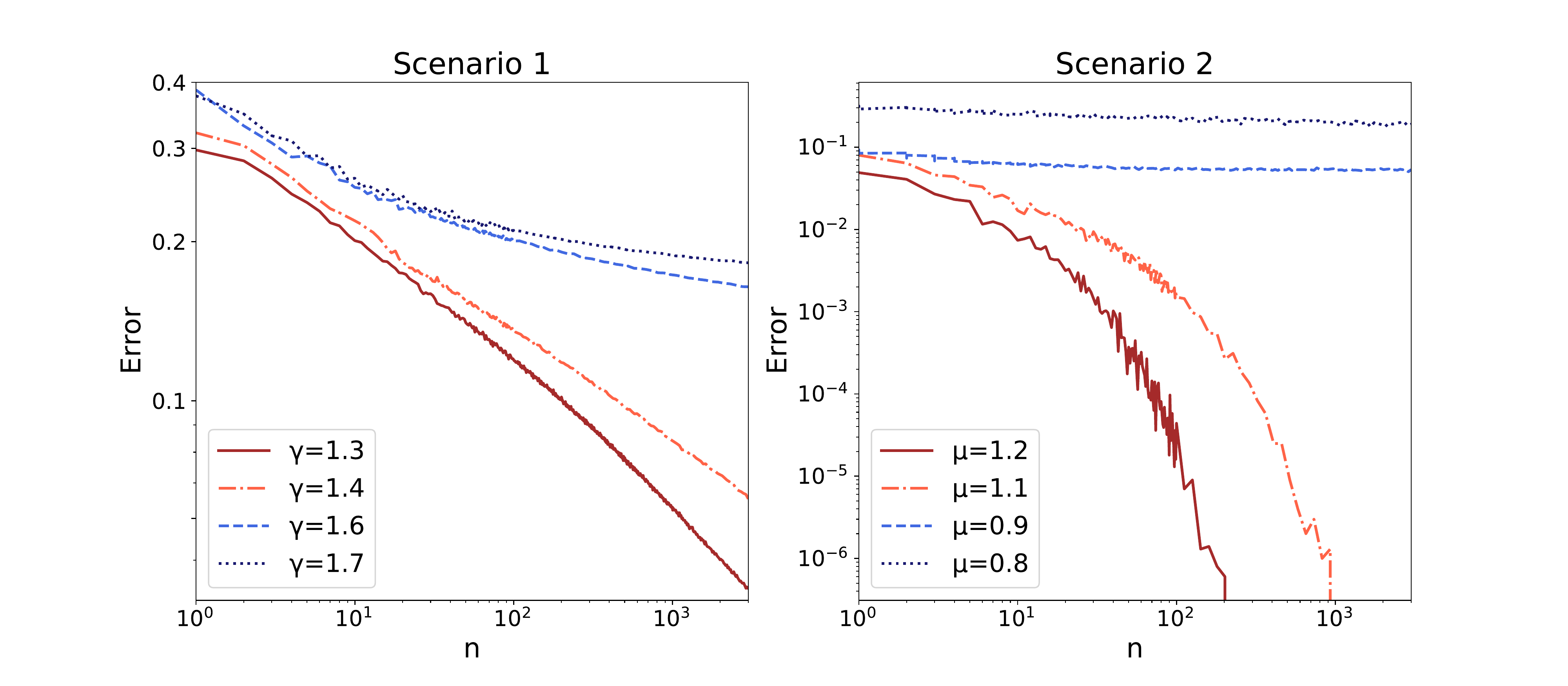}
  \caption{Error (logarithm of misclassification error rate) by the RKHS against $\log n$.
  Left: Scenario 1 for the Delaigle--Hall condition with $\gamma \in \{1.3\,\mathrm{(solid)},1.4\,\mathrm{(dashes)},1.6\,\mathrm{(dots)},1.7\,\mathrm{(dotdash)}\}$. Right: Scenario 2 for the hard-margin condition with $\mu \in \{1.2\,\mathrm{(solid)},1.1\,\mathrm{(dashes)},0.9\,\mathrm{(dot)},0.8\,\mathrm{(dotdash)}\}$.
\label{fig:proposed}}
\end{figure}

We achieve the following findings from the results:
(i) In Scenario 1 for the Delaigle--Hall condition, the error curves show slight differences in shape as well as slope.
That is, the curves are convex when $\gamma = 1.6$ or $1.7$ (the Delaigle--Hall condition is not satisfied), hence they seem to have a slow convergence.
(ii) In Scenario 2 for the hard-margin condition, the error curves show fast convergence only when $\mu=1.2$ and $1.1$ (the hard-margin condition is satisfied).
The results show that the conditions have an effect on the decay speed or errors, weakly for the Delaigle--Hall condition and drastically for the hard-margin condition.

We investigate an effect of a bandwidth election on the results. 
Specifically, we consider Scenario 1 and set the bandwidth to 10, 50, and 100, repeated each simulation 200 times, and calculated the average classification error.
The results are shown in Figure \ref{fig:bandwidth}.
As the bandwidth increases, the decay of errors becomes more gradual.
When bandwidth is large, the perfect classification does not hold, because the expressive power of the kernel is reduced.
Therefore, regardless of the value of $\gamma$, the exponential decay of errors becomes harder to hold.

\begin{figure}[H]
\centering
  \hspace*{-1cm}
  \includegraphics[width=0.75\textwidth]{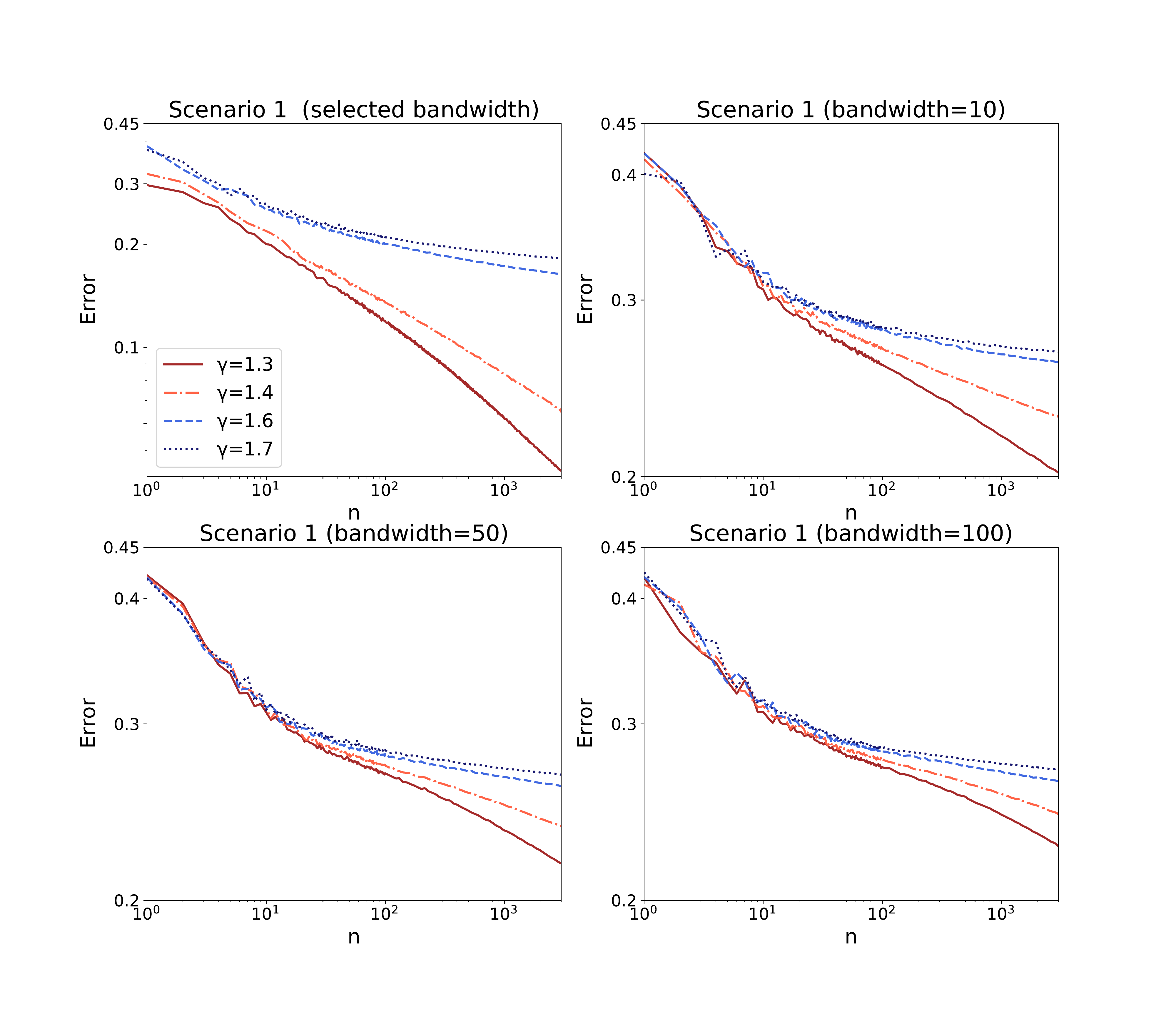}
  \caption{Error (logarithm of misclassification error rate) by the RKHS against $\log n$.
  Upper left: Scenario 1 with bandwidth selected by CV.
  Other three: Scenario 1 for bandwidth $\lambda=10,50,100$. 
  \label{fig:bandwidth}}
\end{figure}

\subsection{Others Methods and the Delaigle--Hall / hard-margin Condition}
We compare the misclassification errors of several common classification methods for functional data.
We consider the following classifiers: (a) the kernel classifier \citep{dai2017optimal}; (b) centroid method \citep{delaigle2012achieving}; (c) centroid method with partial least square (PLS) \citep{preda2007pls}; (d) logistic regression with Gaussian process (GP); and (e) linear discriminant analysis (LDA). The hyper-parameters in (b) anc (c) are chosen by the same way as \cite{delaigle2012achieving}. The bandwidth of the kernel in (a) and the number of components for dimension reduction in (e) are selected by CV. The hyper-parameters in (d) are optimized by Algorithm 5.1 in \cite{ki2006gaussian}.
We set $n$ from $5$ to $1000$.
The rest of the settings of the data generating process and the RKHS method are the same as those of the previous sections.

The results are shown in Figure \ref{fig:comparison}: the left column is for the Scenario 1 with $\gamma=1.3,\,1.4,\,1.6$ and $1.7$ and the right is for the Scenario 2 with $\mu=1.2,\,1.1,\,0.9$ and $0.8$. 
In Scenario 1 (left column), the parameter $\gamma$ does not have a significant impact on the curves, although the RKHS method shows a slight difference in shape as in the previous section.
In Scenario 2 (right column), the parameter $\mu$ has a significant impact.
As $\mu$ increases and the hard-margin condition is satisfied, the nonlinear methods (RKHS, GP, centroid, and kernel classifier) achieve fast convergence.
In contrast, the linear methods (PLS and LDA) do not.
This finding indicates that the nonlinear methods have the potential ability to achieve fast convergence with the hard-margin condition.

\begin{figure}[tbp]
\centering
  \includegraphics[width=0.65\hsize]{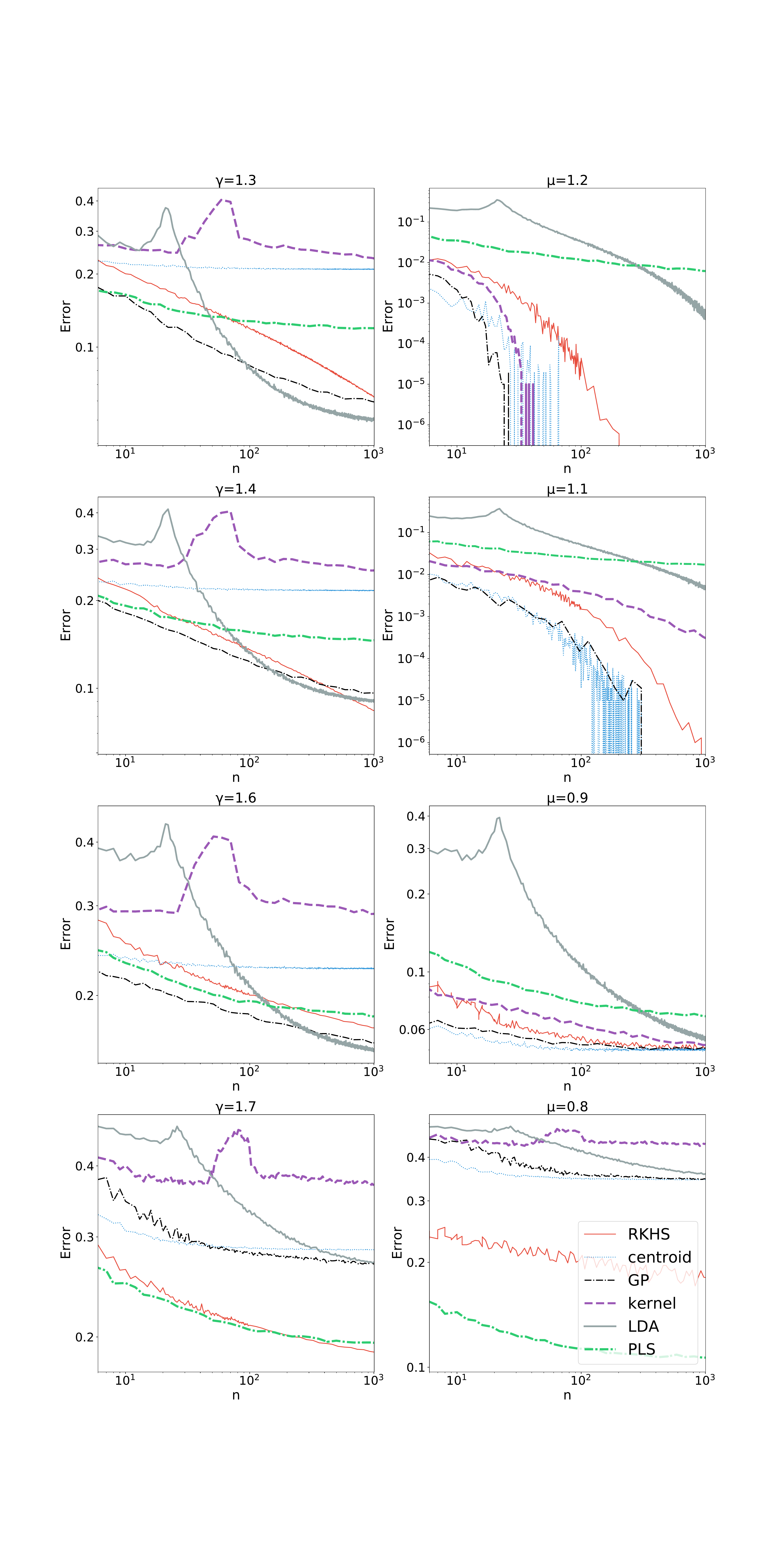}
  \vspace*{-0.7cm}

  \caption{
  Error (logarithm of misclassification error) by the RKHS against $\log n$ of the RKHS method (solid), centroid method (dots), logistic regression with Gaussian process (dashdot), kernel classifier (bold dash), linear discriminant analysis (bold solid), centroid method with partial least square (bold dashdot).
  Left column: Scenario 1 with $\gamma \in \{1.3,1.4,1.6,1.7\}$. 
  Right: Scenario 2 for the hard-margin condition with $\mu \in \{1.2,1.1,0.9,0.8\}$. \label{fig:comparison}}
\end{figure}

\section{Conclusion and Discussion} \label{sec:conclusion}

In this study, we investigate the convergence rate of the misclassification error of the classification problem for functional data, and discuss the feasibility of a small error with finite samples.
The Delaigle--Hall condition guarantees the existence of a perfect classifier, which is a specific condition for functional data that cannot occur for finite-dimensional data. 
However, the minimax rate of misclassification error with functional data, that is, the worst-case error, follows logarithmic convergence in sample size, hence it was not clear whether we can enjoy the perfect classification in practice with a realistic sample size.
Our result reveals that the DH condition leads not only to the existence of a perfect classifier but also to the exponential convergence of the error.
It indicates that the DH condition is also helpful in the sense of estimation from finite samples. 
This reveals the specific advantage of treating functional data explicitly since the DH condition is specific to infinite-dimensional data.

We note that Assumption \ref{asmp:covering} on a covering number restricts an available class of functional data. 
This is unavoidable as long as we handle the properties of functional data in a uniform way using the notion of metric entropy. 
A hopeful way to avoid this is a spectral decomposition-based approach, such as \cite{hall2007methodology}, which directly deals with the randomness of functional data without entropy.

A feature of this study is that the considered classifier is very typical and different from modern adaptive methods, such as neural networks. 
However, owing to this simplicity, we succeed in clarifying the theoretical properties with the perfect classification. 
Moreover, since an analysis of the adaptive methods is often conducted by extensions of analysis for simple methods, our results may serve as a basis for further analysis.

\section*{Acknowledgements}
The authors would like to thank the anonymous referees, an Associate Editor and the Editor for their constructive comments that improved the quality of this paper.
M.Imaizumi was supported by JSPS KAKENHI (18K18114) and JST Presto (JPMJPR1852).

\bibliography{main}
\bibliographystyle{apecon}

\appendix

\section{Proof of Lemma \ref{lem:classifier}} \label{app:lemma1}

\begin{proof}[Proof of Lemma \ref{lem:classifier}]
$f_0$ minimizes $R(f)$, if $\sgn(f_0(x)) = \sgn (\Pr(Y=1 | x) - \Pr(Y=-1|x))$ is satisfied.
The Radon-Nikodym derivative for $\Pr(Y=1|X)\Pi(X) = \Pr(X|Y=1)\Pr(Y=1) = P_+(X)(1-w)$ in terms of $\Pi$ implies $\Pr(Y=1|x) = (1-w) p_+(x)$.
Similarly, we have $\Pr(Y= -1|x) = w p_-(x)$.
Hence, $f_0$ has the desired property.
\end{proof}

\section{Note on Assumption 1}

We firstly provide additional example on Assumption \ref{asmp:covering}.

\begin{example}[Monotone/Convex Path]
Assume $\mX$ is a set of component-wise monotonic functions from $[0,1]^p$ to $[0,1]$ with $p \geq 2$. 
With $\gamma=2(p-1)$, Assumption \ref{asmp:covering} follows from Theorem 1.1 in \cite{gao2007entropy}.
Alternatively, let $\mX$ be a set of convex functions on $[0,1]^p$ that are uniformly bounded. 
From Theorem 3.1 in \cite{guntuboyina2012covering} with setting $\gamma=p/2$, Assumption \ref{asmp:covering} holds.
\end{example}

\begin{example}[Gaussian Process]
Let $X$ be a Gaussian process on $[0,1]^p$ with a positive even $p$, and assume its covariance function $\Cov(t,t'), t,t' \in [-,1]^d$ is $\Cov(t,t') k_\alpha(\|t-t'\|_2)$ where $k_\alpha$ is  Mat\'ern kernel function ((4.14) in \citet{williams2006gaussian}). 
In this case, with probability $1$, a path of $X$ is in a RKHS whose kernel is $k_{\alpha - p/2}$.
Then, if $\mX$ is a unit-ball of the RKHS in terms of an RKHS norm, we obtain that Assumption \ref{asmp:covering} holds with $\gamma = p/(\alpha - p/2)$.
For details, see Corollary 4.15 in \cite{kanagawa2018gaussian}.
\end{example}

We also present the following result to show the validity of Example \ref{ex:unbouded} on unbounded functions.
\begin{proposition} \label{prop:unbounded}
    Let $\mF$ be the set of functions with the form as in Example \ref{ex:unbouded} with fixed $J \in \N$ and locations $t_1,...,t_J \in [0,1]$.
    Then, there exists a constant $C^*$ such that the following inequality holds for any $\varepsilon \in (0,\bar{\varepsilon})$ with existing $\bar{\varepsilon}$:
\begin{align*}
    \log \mN(\varepsilon, \mW^\alpha, \|\cdot\|_{L^2}) \leq V' \varepsilon^{-1/\alpha},
\end{align*}
\end{proposition}
\begin{proof}[Proof of Proposition \ref{prop:unbounded}]
Let $\mW^\alpha$ be a unit-ball in the Sobolev space on $[0,1]$ with an order $\alpha \in \N$.
By applying Theorem 4.3.26 in \citet{gine2016mathematical}, there exists an constant $V'$ such that the following inequality
\begin{align*}
    \log \mN(\varepsilon, \mF, \|\cdot\|_{L^2}) \leq C^* \varepsilon^{-1/\alpha},
\end{align*}
for every $\varepsilon >0 $.
Hence, we set $M_1 = M_1(\varepsilon) = \log \mN(\varepsilon, \mW^\alpha, \|\cdot\|_{L^2})$ and take a subset $\{g_m\}_{m=1}^{M_1} \subset \mW^\alpha$ as centers of the $\varepsilon$-balls to cover $\mW^\alpha$, that is, for any $g \in \mW^\alpha$, there exists $g' \in \{g_m\}_{m=1}^{M_1}$ such that $\|g - g'\|_{L^2} \leq \varepsilon$.

We also consider a set of location parameters $a_j \in [0,1]$ and a covering number of a parameter space for the locations.
Let $\mI = [0,1]^J$ be the space for $A = (a_1,...,a_J) \in \mI$.
We know that there exists a constant $C > 0$ such that
\begin{align*}
    \mN(\varepsilon, \mI, \|\cdot\|) \leq \mN(\varepsilon, [0,1], \|\cdot\|)^J \leq (C \varepsilon)^{-J}.
\end{align*}
Then, let $\{A_m\}_{m=1}^{M_2}$ be subsets of size $M_2 = M_2(\varepsilon)$ such that there are the centers of the $\varepsilon$-balls to cover $\mI$.

Fix a function $f$ which has the form in Example \ref{ex:unbouded} as
\begin{align}
    f(x; g, A) = g(x) + \sum_{j=1}^J \psi(x; a_j,t_j). \label{def:f_sm}
\end{align}
Note that the locations $t_1,...,t_J \in [0,1]$ are fixed.
By the definition of the subsets, we can find $g_m $ and $ A_{m'}$  from the subsets such that $\|g - g'\|_{L^2} \leq \varepsilon$, and $\|(a_1,...,a_J)^\top - A_{m'}\| \leq \varepsilon$ for each $\varepsilon$.
Then, we define
\begin{align*}
    \hat{f}(x) := g_m(x) + \sum_{j=1}^J \psi(x; a_{m',j}, t_j),
\end{align*}
where we write $A_{m'} = (a_{m',1},...,a_{m',J})^\top$.
We can bound the following difference as
\begin{align*}
    \|f - \hat{f}\|_{L^2} &\leq \|g - g_m\|_{L^2} +  \left\| \sum_{j=1}^J\psi(\cdot; a_j,t_j) - \sum_{j=1}^J \psi(\cdot; a_{m',j},t_j)\right\|_{L^2} \\
    &\leq \varepsilon +\sum_{j=1}^J \left\| \psi(\cdot; a_j,t_j) -  \psi(\cdot; a_{m',j},t_j)\right\|_{L^2}.
\end{align*}
About the norm in the last term, we can bound it as
\begin{align*}
    &\left\| \psi(\cdot; a_j,t_j) -  \psi(\cdot; a_{m',j},t_j)\right\|_{L^2}^2 \\
    & \leq  \int_0^1 \left( \frac{a_j}{|x-t_j|^{1/3}} - \frac{a_{m',j}}{|x-t_j|^{1/3}}  \right)^2 dx \\
    & = (a_j - a_{m',j})^2 \int_0^1 \left(  \frac{1}{|x-t_j|^{1/3}}  \right)^2 dx \\
    &=(a_j - a_{m',j})^2 3 ((1-t_j)^{1/3} + t_j^{1/3})\\
    & \leq 6 (a_j - a_{m',j})^2.
\end{align*}
Combining the results and the Cauchy-Schwartz inequality, we obtain
\begin{align*}
    \|f - \hat{f}\|_{L^2}&\leq \varepsilon + \sqrt{6}\sum_{j=1}^J |a_j - a_{m',j}|  \leq \varepsilon + \sqrt{6} \sqrt{J} \|A - A_{m'}\| \leq (1 + \sqrt{6J}) \varepsilon
\end{align*}
Hence, we find that the product set of $\{g_m\}_{m=1}^{M_1}$ and $\{A_m\}_{m=1}^{M_2}$ can construct a $(1 + \sqrt{6J}) \varepsilon$-covering set of a set of $f$ with the form \eqref{def:f_sm}.
Then, we bound the covering number of $\mF$ as
\begin{align*}
    \log \mN((1 + \sqrt{6J}) \varepsilon, \mF, \|\cdot\|_{L^2}) &\leq \log \mN(\varepsilon, \mF, \|\cdot\|_{L^2}) + \log \mN(\varepsilon, \mI, \|\cdot\|) \\
    &\leq  V' \varepsilon^{-1/\alpha} + J \log (C\varepsilon^{-1}).
\end{align*}
We update $\varepsilon$ as $\varepsilon \leftarrow (1 + \sqrt{6J}) \varepsilon$ and achieve $C^*$ such that we can ignore the term with the order of $\log(1/\varepsilon)$, then obtain the statement.
\end{proof}

\section{Note on Assumption 2}

Several distributions are known to satisfy Assumption \ref{asmp:smallball}.
We develop the following simple example:
\begin{example}[Uniformly distributed Fourier coefficients]
    We consider a distribution $\Pi$ of a function $h$ on $[0,1]$ whose Fourier coefficients by a basis are uniformly distributed.
    Let $\{\varphi_j: [0,1] \to \R \}_{j = 1,2,...,\infty}$ be a trigonometric basis as an orthonormal basis (see Example 1.3 in \cite{tsybakov2008introduction}).
    We set $\Pi$ as a measure of $h$ which has a form
    \begin{align*}
        h(\cdot) = \sum_{j=1}^\infty \theta_j \varphi_j(\cdot),
    \end{align*}
    where $\theta_j$ is a random Fourier coefficient which independently follows a uniform distribution on $[-1/j, 1/j]$.
    Note that Parseval's equality yields $\|h\|_2^2 = \sum_{j=1}^\infty \theta_j^2 \leq \sum_{j=1}^\infty 1/j^2 = \pi^2/6$ almost surely, hence the support of $\Pi$ is in the $L^2$ space.
    Furthermore, $h$ belongs to  the Sobolev space since the coefficients $\{\theta_j\}_{j=1}^\infty$ are in the Sovolev ellipsoid (for details, see Section 1.7.1 in \cite{tsybakov2008introduction}), the support of $\Pi$ satisfies Assumption \ref{asmp:covering}.

    We show that $\Pi$ satisfies Assumption \ref{asmp:smallball}.
    Without loss of generality, we consider a ball $B(0,\delta)$ whose center is $0$ with fixed $\delta > 0$.
    We define $C_{1.5} := \sum_{j=1}^{\infty} 1/j^{1.5} \approx 2.61238$.
    We study the measure as
    \begin{align*}
        \Pi (h \in B(0,\delta)) &= \Pi(\|h\|^2 \leq \delta^2) \\
        &= \Pi \left(\sum_{j=1}^\infty \theta_j^2 \leq \delta^2 \right) \\
        &= \Pi \left(\sum_{j=1}^\infty \theta_j^2 \leq \frac{\delta^2}{C_{1.5}} \sum_{j=1}^\infty  j^{- 1.5}\right) \\
        & \geq \prod_{j=1}^\infty \Pi \left(\theta_j^2 \leq \frac{\delta^2}{C_{1.5} j^{1.5}}   \right) \\
        & = \prod_{j=1}^J \Pi \left(\theta_j^2 \leq \frac{\delta^2}{C_{1.5} j^{1.5}}   \right),
    \end{align*}
    where $J = \max\{j \in \N \mid 1/j^2 \geq \delta^2/ (C_{1.5} j^{1.5})\}$.
    The first inequality follows the independent property of $\theta_j$, and the last equality follows that $\Pi (\theta_j^2 \leq \frac{ \delta^2}{C_{1.5} j^{1.5}}   ) = 1$ for $j \geq J + 1$.
    For $j \leq J$, $\Pi (\theta_j^2 \leq \frac{\delta^2}{C_{1.5} j^{1.5}} )$ is positive since $ \theta_j$ follows the uniform distribution, we obtain that $\Pi (h \in B(0,\delta)) > 0$.
    \qed
\end{example}

Another example is the truncated Gaussian as described below.
\begin{example}[Small shifted ball probability with truncated Gaussian processes]\label{example:shift}
Let $h$ be a Borel measurable centered Gaussian random element in a separable Hilbert space $(\mathbb{H},\|\cdot\|_{\mathbb{H}})$. 
From \cite{kuelbs1994gaussian}, for any $x \in \mathbb{H},\, \varepsilon>0,\,0\leq \alpha\leq 1$, it holds that
\begin{align}\label{ineq:smball}
    \Pi(h: \|h-x\|_{\mathbb{H}}\leq \varepsilon)\geq \exp\left\{-\inf_{x_0:\in\mathbb H :\|x_0-x\|\leq \alpha\varepsilon}\frac{\|x_0\|_{\mathbb{H}}^2}{2}+\log \Pi(\|h\|_{\mathbb{H}}<(1-\alpha)\varepsilon)\right\}.
\end{align}
To satisfy Assumption \ref{asmp:covering}, we consider a probability measure of a truncated version of a Gaussian measure.  Given a constant $c > 0$ as a truncation level, we define a ball $\mathbb{H}_c := \{h \in \mathbb{H} \mid \|h\|_{\mathbb{H}} \leq c\}$ such that $\Bar{C} := \Pi(\mathbb{H}_c) > 0$. 
We, then, consider a measure $\Pi_c$ associated with the truncated Gaussian process, such that $H \in \sigma(\mathbb{H}_c)$ satisfies $\Pi_c(H) := \Pi(H \mid \mathbb{H}_c) =  \Pi(H) / \Bar{C}$.
Using the inequality (\ref{ineq:smball}), for any $x \in \mathbb{H}_c$, it holds that
\begin{align*}
    &\Pi_c(h : \|h-x\|_{\mathbb{H}}\leq \varepsilon) \\
    &\geq \exp\left\{-\inf_{x_0:\in\mathbb H :\|x_0-x\|\leq \alpha\varepsilon}\frac{\|x_0\|_{\mathbb{H}}^2}{2}+\log \Pi(\|h\|_{\mathbb{H}}<(1-\alpha)\varepsilon)\right\} \Bar{C}^{-1},
\end{align*}
for any $\alpha \in (0,1)$.
Hence, by setting $\mathbb{H}_c=L^2,\ \alpha=\frac{1}{2}$ and $\varepsilon=\frac{\delta}{2}$, we obtain 
\begin{align*}
    \Pi_c \left(B\left(x;\frac{\delta}{2}\right)\right) &=\Pi\left(h : \|h-x\|_{L^2}\leq \frac{\delta}{2}\right)\Bar{C}^{-1}\\
    &\geq \exp\left\{-\inf_{x_0:\in L^2:\|x_0-x\|\leq \delta/4}\frac{\|x_0\|_{L^2}^2}{2}+\log \Pi\left(\|h\|_{L^2}<\frac{\delta}{4}\right) \right\}\Bar{C}^{-1}\\
    &\geq \Pi\left(h: \|h\|_{L^2}<\frac{\delta}{4}\right) \exp\left(-\frac{\|x\|_{L^2}^2}{2}\right)\Bar{C}^{-1} \\
    & \geq \Pi\left(h: \|h\|_{L^2}<\frac{\delta}{4}\right) \exp\left(-\frac{c^2}{2}\right)\Bar{C}^{-1},
\end{align*}
for any $x \in \mathbb{H}_c$.
Since $h$ is a centered Gaussian, a ball near $0$ with positive radius has positive measure \citep{gao2004exact}. Then $\Pi (B(x;{\delta} / {2}))>0$ holds.
\qed
\end{example}

\section{Proof of the Delaigle--Hall and hard-margin Condition} \label{app:DH-and-HM}

We start with the proof for connecting the Delaigle--Hall condition and the hard-margin condition, which is one of the key points of this study.

\begin{proof}[Proof of Proposition \ref{prop:dhhm}]
We will develop an explicit classifier based on the Delaigle--Hall condition, then show that the classifier has a positive margin.
Without loss of generality, we set $\mu_- = 0$, hence $\mu_{-,j} = 0$ holds for all $j \in \N$.
Hence, we have $\psi_M:=\sum_{j=1}^M \theta_j^{-1}\mu_{+,j}\phi_j$, and 
    $f_M^*(x) = (\langle x,\psi_M \rangle-\langle \mu_+,\psi_M \rangle)^2-\langle x,\psi_M \rangle^2$ for $x \in \mX$ and $M \in \N$.
For $X$ generated from $P_-$, $f_M^*(X)$ is written as
\begin{align*}
    f_M^*(X)=\langle \mu_+,\psi_M \rangle^2- 2\langle \mu_+,\psi_M \rangle \alpha_- Z_-,
\end{align*}
where the random variable $Z_- = \langle X,\psi_M\rangle /\alpha_-$ and $\alpha_-^2=\sum_{j=1}^{\infty}\theta_{-,j}\langle \psi_M, \phi_{-,j} \rangle^2$. 
Here, $E[Z_-]=0$ and $E[Z_-^2]=1$ hold.
Similarly, for $X$ generated from $P_+$,  we obtain
\begin{align*}
    f_M^*(X) &=-\langle \mu_+,\psi_M \rangle^2 - 2\langle x-\mu_+,\psi_M \rangle \langle \mu_+ ,\psi_ M\rangle =-\langle \mu_+,\psi_M \rangle^2 - 2\langle \mu_+,\psi_M \rangle \alpha_+ Z_+ ,
\end{align*}
where $Z_+ = \langle X-\mu_+,\psi_M\rangle /\alpha_+$ 
and $\alpha_+^2=\sum_{j=1}^{\infty}\theta_{+,j}\langle \psi_M, \phi_{+,j} \rangle^2$. 
Here, $Z_+$ satisfies $E[Z_+]=0$ and $E[Z_+^2]=1$.

Now, we evaluate the margin of the classifier $f^*_M$ with the measure $\Pi$. For any $\delta>0$, we bound it as
\begin{align*}
    &\Pi(\{ x:|\, \|x-\mu_+\|^2-\|x\|^2| \leq \delta \})\\ &=\lim_{M\rightarrow\infty} \Pi(\{ x:| f_M^*(x)|\leq \delta \})\\
    &= \lim_{M\rightarrow\infty} w P_-(|f_M^*(X)|\leq\delta)+(1-w)P_+(|f_M^*(X)|\leq \delta) \\
    &\leq  \lim_{M\rightarrow\infty} w P_-(f_M^*(X)\leq \delta)+(1-w)P_+(f_M^*(X)\geq-\delta) \\
    &= \lim_{M\rightarrow\infty} w P_-( \langle \mu_+,\psi_M \rangle^2- 2\langle \mu_+,\psi_M \rangle \alpha_- Z_- \leq \delta) \\
    &\,\,\,\,\,\,+(1-w)P_+( -\langle \mu_+,\psi_M \rangle^2 - 2\langle \mu_+,\psi_M \rangle \alpha_+ Z_+\geq -\delta) \\
    &= \lim_{M\rightarrow\infty} w P_-\left( Z_-\geq \frac{\langle \mu_+ ,\psi_M \rangle^2-\delta}{2\alpha_-\langle \mu_+,\psi_M \rangle}\right) +(1-w)P_+\left(-Z_+\geq\frac{\langle \mu_+,\psi_M \rangle^2-\delta}{2\alpha_+\langle \mu_+,\psi_M \rangle}\right) \\
    &\leq \lim_{M\rightarrow\infty} \frac{ \{4w\alpha_-^2+4(1-w)\alpha_+^2 \} \langle \mu_+,\psi_M \rangle^2 }{(\langle \mu_+,\psi_M \rangle^2-\delta)^2}\ \ \ \ (\because \text{Chebyshev's inequality})\\
    &=0.
\end{align*}
The last equality holds because of the following relation: for $\ell\in \{-,+\}$, we obtain 
\begin{align*}
    \lim_{M\rightarrow\infty} \frac{\langle \mu_+,\psi_M \rangle^2}{\alpha_{\ell}^2}&=\lim_{M\rightarrow\infty}\frac{(\sum_{j=1}^M\theta_{ j}^{-1}\mu_{j}^2)^2}{\sum_{j=1}^{\infty}\theta_{\ell, j} \langle \psi_M, \phi_{\ell, j} \rangle^2}\\
    &=\lim_{M\rightarrow\infty}\frac{(\sum_{j=1}^M\theta_j^{-1}\mu_{j}^2)^2}{\sum_{j=1}^{\infty}\theta_{\ell, j} (\sum_{i=1}^M \theta_i^{-1}\mu_i \langle \phi_i,\phi_{\ell, j} \rangle)^2}\\
    &=\infty,
\end{align*}
by the Delaigle--Hall condition.
\end{proof}

\section{Proof of Convergence Analysis} \label{app:convergence}

\subsection{Additional Notation}

For a function $f: \mX\times \{-1,1\} \to \R$, we employ the notation $(\ell \circ  f)(x,y)=\ell(yf(x))$.
Also, for $g = \ell \circ  f$, its expectation and empirical mean with respect to $P$ is written as $Pg =\mathbb{E}_{(X,Y) \sim P}[g(X,Y)]$ and $P_n f=\frac{1}{n}\sum_{i=1}^n g(X_i,Y_i)$ with the observed data $\{(X_i,Y_i):i=1,...,n\}$. 

We define an open ball $B(x; \delta') \subset \mX$ of radius $\delta'$ centered at $x \in \mX$ with metric $\|\cdot\|$. 
We also define a set $\mH(x,\delta') \subset \mH$ which is a set of a map $h \in \mH$ satisfying the following three conditions: 
\begin{align}
    &(i)\, \forall x'\in \mX\,\,0\leq h(x')\leq 2\delta, ~ (ii)\,h\geq \delta' \mathrm{~on~}B\left(x;\frac{\delta'}{2}\right), \mbox{~and~} \notag \\
    &(iii) \int_{B(x; \delta')^c}hd\Pi\leq \delta'\int_\mX h d\Pi, \label{cond:rkhs}
\end{align}
where $B(x; \delta')^c := \mX \backslash B(x; \delta')$.
It is obvious to show $\mH(x,\delta')\neq \emptyset$, since there exists a continuous $f$ such that $0\leq f \leq \frac{3}{2}\delta'$ on $B(x,\delta'/2)$ and  $f=0$ on $B(x,\delta')^c$ holds, and $\mH$ is dense in $C(\mX)$.

We define $q(x,\delta')=\inf_{h\in \mH (x,\delta')}\|h\|_{\mH}$ and $\Bar{q}(\delta')$ as its decreasing envelope such that $\bar{q}(\delta') \geq \sup_{x \in \mX} q(x,\delta') $ holds.
We also define $p(x,\delta') =(\delta')^2 \Pi (B(x;{\delta'} / {2}))$ and define its lower envelope function $\Bar{p}$ as $p(x,\delta')\geq \Bar{p}(\delta')>0$ for all $x$ such that $|\tilde{f}^*(x)|\geq 1$ holds.
This definition is related to the small shifted ball probability and it varies with the setting of $\Pi$ and $\mX$. Remark that the existence of a lower envelope is guaranteed by Assumption \ref{asmp:smallball}.
Further, on the set $\{x:|\tilde{f}^*(x)|\geq 1\}$, we consider a positive function $r: \R_+ \to \R_+$ such that $r(\delta')\geq \Bar{p}(\delta')/\Bar{q}(\delta') > 0$ holds.

\subsection{Full Proof}

\begin{proof}[Proof of Theorem \ref{thm:convergence}]

This proof contains three steps: (i) a basis decomposition on the generalization error, (ii) bound a misclassification error with the bounded condition, and (iii) bound an unbounded probability.
In the following, each step is described in one subsection.

\textbf{(i) Basic Decomposition}:
We start with a basic decomposition for the generalization error for the classification.
To fit the situation with the Delaigle--Hall condition, we extend its formulation.
In the following, $\Pr(\cdot)$ and $\Ep[\cdot]$ denote a probability and an expectation with respect to the observed data from $P^{\otimes n}$.

\begin{lemma} \label{lem:decomp_generror}
    Suppose the Delaigle--Hall condition holds.
    Then, the following equation holds:
    \begin{align*}
        \Ep[R(\hat{f}_n) - R(\tilde{f}^*)] \leq \int |\eta(x)| \, \Pr(\hat{f}_n(x) \tilde{f}^*(x)\leq 0) \,d\Pi(x).
    \end{align*}
\end{lemma}
\begin{proof}[Proof of Lemma \ref{lem:decomp_generror}]
We transform the generalization error for any $f \in \mH$ as 
\begin{eqnarray*}
    R(f)-R(\tilde{f}^*)&=&\mathrm{E}_X[\mathrm{E}_{Y}[ I_{\{Y\neq \sgn(f(X))\}}-I_{\{Y\neq \sgn(\tilde{f}^*(X))\}} \, |\,X]]\\
    &=&\mathrm{E}_X[ \{I_{\{1\neq \sgn(f(X))\}}-I_{\{1\neq \sgn(\tilde{f}^*(X))\}}\}\cdot \Pr(Y=1|X) \\ &\,&\,\hspace{20pt} +\{I_{\{-1\neq \sgn(f(X))\}}-I_{\{-1\neq \sgn(f^*(X))\}}\}\cdot\Pr(Y=-1|X)\,]\\
    &\leq&\mathrm{E}_X[ I_{\{\sgn(\tilde{f}^*(X))\neq \sgn(f(X))\}}\,|\eta(X)| \,]\\
    &=&\int_{\{x \in \mX: \sgn(\tilde{f}^*(x))\neq \sgn(f(x))\}}\,  |\eta(x)|\, d\Pi(x).
\end{eqnarray*}
We consider its expectation with $\hat{f}_n$ and develop its upper bound as
\begin{eqnarray*}
    \Ep[R(\hat{f}_n)-R(\tilde{f}^*)]&=&\Ep\left[\int_{\{x \in \mX : \sgn(\tilde{f}^*(x))\neq \sgn(\hat{f}_n(x))\}} |\eta(x)| d\Pi(x)\,\right]\\
    &=& \Ep\left[\int_{\{x \in \mX: \hat{f}_n(x) \tilde{f}^*(x)\leq 0 \}} |\eta(x)| d\Pi(x)\right]\\
    &=&\int |\eta(x)| \, \Ep[\,I_{\{ x \in \mX: \hat{f}_n(x) \tilde{f}^*(x)\leq 0\}}\,]\, d\Pi(x)\\
    &=&\int |\eta(x)| \, \Pr(\hat{f}_n(x)\tilde{f}^*(x)\leq 0) \,d\Pi(x).
\end{eqnarray*}
Then, we obtain the statement.
\end{proof}

Our next goal is to study the probability $\Pr (\hat{f}_n(x) \tilde{f}^*(x)\leq 0)$ in Lemma \ref{lem:decomp_generror} for a given $x \in \mX$.
For any $x \in \mX$ such that $\tilde{f}^*(x)>0$ holds, with the threshold $U$, we obtain
\begin{align}
    \Pr(\hat{f}_n(x) \tilde{f}^*(x)\leq 0)&=\Pr(\hat{f}_n(x)\leq 0) \notag \\
    &\leq \underbrace{\Pr(\hat{f}_n(x)\leq 0\,,\,\|\hat{f}_n\|_\mH\leq U)}_{= T_1}+ \underbrace{\Pr(\|\hat{f}_n\|_\mH>U)}_{= T_2} \label{ineq:basic_decomp}.
\end{align}
If $\tilde{f}^*(x) < 0$ holds, we obtain the similar bound.
We will bound the terms $T_1$ and $T_2$, respectively.

\textbf{(ii-1) Bound $T_1$ via hard-margin Condition}:
As preparation, we fix $x$ such that $\tilde{f}^*(x) \geq \delta = 1$ holds, which follows from $\mathrm{ess}\inf_{x \in \mX}|\tilde{f}^*(x)| \geq \delta$ for any $\delta$ by Proposition \ref{prop:dhhm}.
Also, we fix $\delta_0 > 0$ then pick $h \in \mH(x,\delta_0)$ as \eqref{cond:rkhs}.
We rewrite the empirical loss in \eqref{def:erm} as
\begin{align*}
    L_n(f) = \frac{1}{n}\sum_{i=1}^n \ell(Y_i f(X_i)) + \lambda \|f\|_\mH^2.
\end{align*}
By Lemma \ref{lem:func_derive}, we obtain its functional derivative in terms of $f$ at $\hat{f}_n$ with direction $h$ as
    $\nabla L_n(\hat{f}_n) = \frac{1}{n}\sum_{i=1}^n\ell'(Y_i\hat{f}_n(X_i))Y_ih(X_i)+2\lambda\langle\hat{f}_n,h\rangle_\mH$.
By the optimal condition of $\hat{f}_n$, we have $\nabla L_n(\hat{f}_n) = 0$.

We bound the term $T_1$ by combining a probability of the event with $\nabla L_n(\hat{f}_n)$.
Let $\mU$ be an event $\{ \hat{f}_n(x)\leq 0, \|\hat{f}_n\|_\mH \leq U\}$.
We simply obtain
\begin{align*}
    T_1 &= \Pr(\mU, \nabla L_n(\hat{f}_n) = 0 )\\
    &= \Pr(\nabla L_n(\hat{f}_n) = 0 ~|~ \mU ) \Pr (\mU )  \\
    & = \{ 1 - \Pr(  \nabla L_n(\hat{f}_n) \neq 0 ~|~ \mU )\} \Pr (\mU ) \\
    & \leq \{1 - \Pr(  \nabla L_n(\hat{f}_n) < 0 ~|~ \mU) \} \Pr(\mU)\\
    & \leq 1 - P_{L},
\end{align*}
where we define $P_L = \Pr(  \nabla L_n(\hat{f}_n) < 0 ~|~ \mU)$.
The first line follows the fact $\Pr(\nabla L_n(\hat{f}_n) = 0) = 1$.
To bound $T_1$, we will study $P_{L}$.

We consider an event $\mU$, and study the derivative $\nabla L_n(\hat{f}_n)$.
We define $ \nabla \hat{L} = \frac{1}{n}\sum_{i=1}^n\ell'(Y_i\hat{f}_n(X_i))Y_ih(X_i)$ as a derivative of the loss function part from $\nabla L_n(\hat{f}_n)$.
By Lemma \ref{lem:bound_xi} associated with Lemma \ref{lem:bern_bound}, we can bound tail probability of $\hat{L}$ as
\begin{align*}
    \Pr \left( \nabla \hat{L}<-\frac{1}{2}\delta_0\Ep[h(X)] ~|~ \mU \right) & \geq  1 - 2\exp\left(-\frac{n\delta_0\Ep [h(X)]}{C_{L,U}}\right) \\
    &\geq 1-2\exp\left(-\frac{np(x,\delta_0)}{C_{L,U}}\right),
\end{align*}
which follows the relation $\delta_0\Ep [h(X)]\geq \delta_0^2\Pi(B(x;\delta_0/2))=p(x,\delta_0)$.
By this result, we can also bound $\nabla L_n(\hat{f}_n)$ as
\begin{align*}
   \nabla L_n(\hat{f}_n) &= \nabla \hat{L} +  2\lambda\langle\g,h\rangle \\
   &\leq -\delta_0\Ep[h(X)]/2 +2\lambda U\|h\|_\mH\\
   &\leq -p(x,\delta_0)/2 +2\lambda Uq(x,\delta_0),
\end{align*}
with probability at least $1-2\exp(-np(x,\delta_0)/{C_{L,U}})$.
The first inequality follows the Cauchy-Schwartz inequality and $\|\g\|_\mH \leq U$. 
Since we set $\lambda <\frac{p(x,\delta_0)}{4Uq(x,\delta_0)}=\frac{r(x,\delta_0)}{4U}$, we obtain $\nabla L_n(\hat{f}_n) < 0$ with the probability.
Thus, we have 
\begin{align}
      T_1 &\leq 1- P_{L} \leq 1 - \left\{ 1-2\exp\left(-\frac{np(x,\delta_0)}{C_{L,U}}\right)\right\} \notag \\
      &\leq 2\exp\left(-\frac{np(x,\delta_0)}{C_{L,U}}\right) \leq 2\exp\left(-\frac{n\bar{p}(\delta_0)}{C_{L,U}}\right). \label{ineq:T1}
\end{align}

\textbf{(ii-2) Bound $T_2$ via Metric Entropy of Functional Data Space}:
We bound $T_2$ in \eqref{ineq:basic_decomp} by using the \textit{peeling} technique (for introduction, see Chapter 7 in \cite{steinwart2008support}).

As preparation, we derive an upper bound of $\|\g\|_\mH$.
Since 
\begin{align*}
    \lambda \|\hat{f}_n\|_\mH^2 \leq P_n(\ell\circ \hat{f}_n)+\lambda \|\hat{f}_n\|_\mH^2\leq \ell(0) + \|0\|_\mH^2 \leq 1,
\end{align*}
where the second inequality is obtained by replacing $\g$ by $0$ and the optimality condition of $\g$, and the last inequality follows the bounded condition on the loss function, we obtain $\bar{R} = \lambda^{-1/2}\ell(0)^{-1/2}$ as an upper bound of $\|\hat{f}_n\|_\mH$.

We decompose the term $T_2$.
We remind the definition $f^\dagger = \argmin_{f \in \mH} R(f)$, and consider a constant $R = \|f^\dagger\|_\mH$ which is assumed to be no less than $1$ without loss of generality.
We also define events $\mA(R)$ and $\mE(R)$ as
\begin{align*}
    \mA(R) = \left\{ \frac{R}{2} \leq  \|\g\|_\mH \leq R \right\}, \mbox{~and~} \mE(R) = \left\{ \|\g\|_\mH \leq \frac{R}{2} \right\},
\end{align*}
and a sequence $R_k = 2^k, k = 1,2,...,N$ where $N=\log_2\bar R+1$.
For each $\lambda >\underline{\lambda}$ and sufficiently large $n$, we have 
\[ 
N= \log_2\bar R+1 =\frac{1}{2\log2}\log\frac{\ell(0)}{\lambda} +1 \leq \frac{1}{2\log2}\log\frac{n}{C_{V,\gamma} \log \log n} +1 \leq C_{V, \gamma}'\log n,
\]
where  $C_{V, \gamma}'$ is a constant depending on $C_{V, \gamma}$.
We remark that $\cup_{k=1}^N \mA(R_k) \supset \{U \leq \|\g\|_\mH\}$ since $\|\g\|_\mH \leq\bar{R}$ holds.
Since $\mA(R_k),k=1,...,N$ are disjoint up to null sets, we obtain 
\begin{align}
    T_2 & \leq \sum_{k=1}^N \Pr (U \leq \|\g\|_\mH| \mA(R_k)) \Pr (\mA(R_k)) \leq \sum_{k=1}^N \Pr (U \leq \|\g\|_\mH| \mA(R_k)). \label{ineq:T2_1}
\end{align}
Now, we will bound the probability $\Pr (U \leq \|\g\|_\mH| \mA(R_k))$ in the following.

We investigate the event $\mE(R)$ with conditional on $\mA(R)$ and study the event $U \leq \|\g\|_\mH$.
We set a constant $c_{V,\gamma} =  2\sqrt{R}(\sqrt{6}+\frac{1}{V3^{\gamma}})\exp(V3^{\gamma})$.
An inequality $P_n(\ell\circ \hat{f}_n)-\inf_{g \in \mH:\|g\|_\mH\leq R}P_n(\ell\circ g) \geq 0$ and a uniform bound defined by
\begin{align}
    \Delta(n,\gamma,t,R) = {Rc_{V,\gamma}}{ (\log n)^{-1/\gamma} } + \sqrt{2t/n} \label{def:uniform_bound},
\end{align}
and Lemma \ref{lem:erm_ineq} implies
\begin{align*}
    \lambda\|\hat{f}_n\|_\mH^2&\leq P_n(\ell\circ \hat{f}_n)-\inf_{g \in \mH:\|g\|_\mH\leq R}P_n(\ell\circ g)+\lambda\|\hat{f}_n\|_\mH^2\\
    &=\inf_{f \in \mH : \|f\|_\mH\leq R}\left\{ P_n(\ell\circ f)-\inf_{g \in \mH : \|g\|_\mH\leq R}P_n(\ell\circ g)+\lambda\|f\|_\mH^2 \right\}\\
    &\leq \inf_{f \in \mH : \|f\|_\mH\leq R} \left\{ P(\ell\circ f)-\inf_{g \in \mH : \|g\|_\mH\leq R}P(\ell\circ g) +\lambda\|f\|_\mH^2  \right\} + 2 \Delta(n,\gamma,t,R)\\
    &\leq \lambda\|f^\dagger\|_\mH^2 + 2 \Delta(n,\gamma,t,R),
\end{align*}
with probability at least $1-\exp(-t)$ for any $t > 0$.
The last inequality holds by substituting $f^\dagger$.
Combining an inequality $R/2<\|\hat{f}_n\|_\mH $ with this result yields
\begin{align*}
    {R^2}/{4} \leq  \|\g\|_\mH^2 \leq  \|f^\dagger\|_\mH^2+2\Delta(n,\gamma,t,R) / \lambda.
\end{align*}
Solving this inequality with respect to $R$ yields that 
\begin{align*}
    R&\leq {4c_{V,\gamma}}{\lambda^{-1} (\log n)^{-1/{\gamma}}} +\sqrt{ \left({4c_{V,\gamma}}{\lambda^{-1} (\log n)^{-1/{\gamma}}}\right)^2+4\|f^\dagger\|_\mH^2  +{8}{\lambda^{-1}}\sqrt{2t/{n}}        }\\
    &\leq {8c_{V,\gamma}}{\lambda^{-1} (\log n)^{-1/{\gamma}}} +2\|f^\dagger\|_\mH +{2\,(2t)^{1/4}}{\lambda^{-1/2}n^{-1/4} } \\
    &\leq C_{V,\gamma}(\|f^\dagger\|_\mH  \lor {\lambda^{-1} (\log n)^{{-1/\gamma}}} \,\lor \,{t^{1/4}}{\lambda^{-1/2}n^{-1/4} } ),
\end{align*}
where $C_{V,\gamma}$ is a constant depending on $c_{V,\gamma}$.
By setting $t=n\zeta^2$ and with sufficiently small $\zeta > 0$ which will be specified later, we obtain $R\leq C_{V,\gamma}\|f^\dagger\|_\mH  = U$ holds.
Consequently, conditional on $\mA(R)$, the event $\mE(R)$ implies $R \leq U$ with probability at least $1-\exp(-n \zeta^2)$, which contradicts the setting of $R \geq U$.
Hence, for any measurable event $\Omega$, it holds that
\begin{align*}
    \Pr(\Omega \mid \mA(R)) \leq \Pr(\mE(R)^c \mid \mA(R)) \leq 1-(1-\exp(-n \zeta^2)) = \exp(-n \zeta^2).
\end{align*}

We put this inequality with setting $\Omega = \{U \leq \|\g\|_\mH\}$ into \eqref{ineq:T2_1}. Then we obtain
\begin{align}
    T_2 &\leq  \sum_{k=1}^N \Pr(\mE(R_k)^c \mid \mA(R_k)) \leq N \exp(-n \zeta^2)\leq e^{-n\zeta}C_{V,\gamma}'\log n \notag  \\
    & \leq  e^{-n\zeta}C_{V,\gamma}'\exp({n\varepsilon}/{C_{V,\gamma}}) \leq C_{V,\gamma}'\exp\{-n\zeta(1-C_{V,\gamma}^{-1})\} \leq \exp({-n\zeta/2}), \label{ineq:bound_t2_1}
\end{align}
The last third inequality follows the setting of $\zeta$ as following $\zeta \geq C_{V,\gamma}' \geq  C_{V,\gamma}'\frac{\log\log n}{n}$.

\textbf{(iii) Combining Results}:
We utilize the derived bounds for $T_1$ in \eqref{ineq:T1} and for $T_2$ in \eqref{ineq:bound_t2_1} into \eqref{ineq:basic_decomp}, then obtain the following inequality:
\begin{align*}
    \Pr (\hat{f}_n(x) \tilde{f}^*(x)\leq 0) \leq 2\exp\left(-\frac{n\bar{p}(\delta_0)}{C_{L,U}}\right) + \exp\left(-\frac{n \zeta}{2}\right) \leq \exp(-\beta n),
\end{align*}
by selecting $\beta$ depending on $\bar{p}(\delta_0), C_{L,U}$ and $\zeta$.
Finally, we obtain $\Ep[R(\hat{f}_n) - \inf_{f \in \mH} R(f)] \leq \Ep[R(\hat{f}_n) -  R(\tilde{f}^*)]$ by the fact $R(\tilde{f}^*) \leq \inf_{f \in \mH} R(f)$, and the above results yield the statement.
\end{proof}

\section{Entropy Analysis for Functional Data} \label{app:entropy}
We provide several technical results with empirical process techniques.
\begin{lemma}\label{lem:erm_ineq}
    Recall the definition of $\Delta(n,\gamma,t,R)$ in \eqref{def:uniform_bound}.
    For any $f' \in \mH$, we obtain
\begin{align*}
    P_n(\ell\circ f')-\inf_{f: \|f\|_\mH\leq R} P_n(\ell\circ f) \leq P(\ell\circ  f')-\inf_{f: \|f\|_\mH \leq R} P(\ell\circ f) + 2 \Delta(n,\gamma,t,R)
\end{align*}
with probability at least $1-e^t$ with $t > 0$.
\end{lemma}
\begin{proof}[Proof of Lemma \ref{lem:erm_ineq}]
Simply, we obtain
\begin{align*}
    &P_n(\ell\circ f')-\inf_{f: \|f\|_\mH\leq R} P_n(\ell\circ f)\\
    &=P_n(\ell\circ f')-P(\ell\circ f)+P(\ell\circ f')\\
    & \quad -\inf_{f: \|f\|_\mH\leq R} P(\ell\circ f)+\inf_{f: \|f\|_\mH\leq R} P(\ell\circ f)-\inf_{f:\|f\|_\mH\leq R} P_n(\ell\circ f)\\
    &\leq \{P_n(\ell\circ f)-P(\ell\circ f)\}\\
    & \quad +\left\{P(\ell\circ f)-\inf_{f: \|f\|_\mH\leq R} P(\ell\circ f)\right\}+\{P(\ell\circ f^\dagger)-P_n(\ell\circ f^\dagger)\}.
\end{align*}
By Lemma \ref{lem:empirical_bound}, $\Delta(n,\gamma,t,R)$ bounds the last two terms with probability at least $1-t$. 
\end{proof}

To complete Lemma \ref{lem:erm_ineq}, we provide the following lemma.
This result is a well-known result with the Rademacher complexity, but we provide it for the sake of completeness.
We introduce a ball in $\mH$ with radius $R$ as $\mH_R=\{f\in\mathcal{H}:\|f\|_\mH\leq R\}$.
\begin{lemma}\label{lem:empirical_bound}
Define $c_{V,\gamma} =  2\sqrt{R}(\sqrt{6}+\frac{1}{V3^{\gamma}})\exp(V3^{\gamma})$.
For any $t>0$ and any $n \in \N$, we obtain
\begin{align*}
  \Pr\left(\sup_{f\in \mH_R}|P_n(\ell\circ f) - P(\ell\circ f)| \leq Rc_{V,\gamma}{(\log n)^{-1/\gamma}}+\sqrt{{2t}/{n}} \right) \geq 1-\exp({-t}).
\end{align*} 
\end{lemma}
\begin{proof}[Proof of Lemma \ref{lem:empirical_bound}]
We firstly bound the term $\sup_{f\in \mH_R}|P_n(\ell\circ f) - P(\ell\circ f)|$ by its expectation and others.
Since a variation of the term
is at most $2/n$ when one pair of $\{(X_i,Y_i)\}_{i=1}^n$ changes, the McDiarmid's inequality (Theorem 3.3.14 in \cite{gine2016mathematical}) implies that with probability at least $1-e^{-t}$,
\begin{align}
    \sup_{f\in \mH_R}|P_n(\ell\circ f) - P(\ell\circ f)| \leq \Ep\left[\sup_{f\in \mH_R}|P_n(\ell\circ f) - P(\ell\circ f)| \right]+ \sqrt{\frac{2t}{n}}. \label{ineq:fuga}    
\end{align}

Secondly, to bound the expectation term, we define the conditional Rademacher complexity of a class of functions $\mathcal{G}$ as follows:
\begin{align*}
    \mR_n(\mathcal{G})=\frac{1}{n}E_{\sigma}\left[\sup_{f\in \mathcal{G}}\sum_{i=1}^n\sigma_if(x_i)\right],
\end{align*} 
where $\sigma_1,\dots,\sigma_n$ are independent random variables which is $1$ with probability $1/2$ and $-1$ otherwise.
We introduce $(X_1',Y_1'),\cdots,(X_n',Y_n')$ as independent pairs of random variables with the same distribution as $(X,Y)$.
We apply the independent pairs and bound the expectation term in \eqref{ineq:fuga} as
\begin{align*}
    &\Ep\left[\sup_{f\in \mH_R}|P_n(\ell\circ f) - P(\ell\circ f)| \right]\\
    &=\Ep\left[\sup_{f\in \mH_R}\Ep\left[\left|\frac{1}{n}\sum_{i=1}^n\ell(Y_if(X_i)) - \frac{1}{n}\sum_{i=1}^n\ell(Y_i'f(X_i'))\right| \mid \{(X_i,Y_i)\}_{i=1}^n \right]\right]\\
    &\leq \Ep\left[\sup_{f\in \mH_R}\left|\frac{1}{n}\sum_{i=1}^n\ell(Y_if(X_i)) - \frac{1}{n}\sum_{i=1}^n\ell(Y_i'f(X_i'))\right|\right]\\
    &= \Ep_\sigma\left[\sup_{f\in \mH_R}\frac{1}{n}\sum_{i=1}^n \left|\sigma_i \{\ell(Y_if(X_i)) - \ell(Y_i'f(X_i'))\}\right|\right]\\
    &\leq \Ep_{\sigma} \left[\sup_{f\in \mH_R}\frac{1}{n}\sum_{i=1}^n \sigma_i \ell(Y_if(X_i))\right] + \Ep_\sigma \left[\sup_{f\in \mH_R}\frac{1}{n}\sum_{i=1}^n \sigma_i\ell(Y_i'f(X_i'))\right]\\
    &= 2\mR_n(\ell\circ \mH_R),
\end{align*} where $\ell\circ \mH_R=\{\ell \circ f:f\in \mH_R\}$.
The first inequality following the Jensen's inequality, and the third equality follows the distribution equivalence by the random variable $\sigma$.
We further apply the Ledoux-Talagrand contraction inequality (Theorem 3.2.1 in \cite{gine2016mathematical}) with the $1$-Lipschitz continuity of $\ell$ yields $\mR_n(\ell\circ \mH_R)\leq \mR_n(\mH_R)$.
Combining the result with $\eqref{ineq:fuga}$, we obtain
\begin{align}
          \sup_{f\in \mH_R}|P_n(\ell\circ f) - P(\ell\circ f)| \leq 2 \mR_n(\ell\circ \mH_R)+ \sqrt{{2t} / {n}} \leq 2 \mR_n(\mH_R)+ \sqrt{{2t} / {n}}. \label{ineq:hoge}
\end{align}
with probability at least $1-e^{-t}$.

Finally, we apply Lemma \ref{lem:bound_rn} and bound $\mR_n(\mH_R)$.
Then, we obtain the statement.
\end{proof}

To complete the empirical process result, we develop the following covering number result, which is a key term to study the convergence of the classifier with functional data.
\begin{lemma}\label{lem:covering}
There exists a constant $\bar{c} > 0$ such that for any $\varepsilon \in (0,\bar{c})$ such that the following holds:
\begin{eqnarray*}
    \log \mN(\varepsilon,\mH_R,\|\cdot\|_{n}) \leq \frac{6R^2}{\varepsilon}+4R\left[\exp\left\{
V \left(\frac{3R}{\varepsilon}\right)^{\gamma} \right\}-1\right].
\end{eqnarray*}
\end{lemma}

\begin{proof}[Proof of Lemma \ref{lem:covering}]
As preparation, we consider an $\varepsilon$-covering set $\mathcal{X}$ of functions $\{x_1,\dots x_m\}$ for $\mathcal{X}$ with $m=m(\varepsilon)$, that is, for any $x \in \mX$, there exists $x_j$ from the set such that $\|x - x_j\| \leq \varepsilon$ holds.
By Assumption \ref{asmp:covering}, we have $m\leq \exp (c \varepsilon^{-\gamma})$. 
We consider grids in $\mH_R$.
For each $f\in \mH_R$, we define a vector
\begin{align*}
    A_f=\left(\lfloor {f(x_1)}/{\varepsilon} \rfloor \,,\,\lfloor {f(x_2)}/{\varepsilon}\rfloor ,...,\lfloor {f(x_n)}/{\varepsilon}\rfloor\right)^\top \in \R^m.
\end{align*} For any pair $f,g \in \mH_R$ such that $\max_{i=1,...,m}|f(x_i)-g(x_i)|< \varepsilon$ holds, we obtain $A_f = A_g$ since $\lfloor {f(x_i)}/{\varepsilon} \rfloor =\lfloor{g(x_i)}/{\varepsilon}\rfloor$ holds.
We also mention the following difference: for any $x \in \mX$ and $f \in \mH_R$, we obtain
    $|f(x)-f(x_i)|\leq \|f\|_{\mathcal{H}} \|x - x_i\| \leq R \| x - x_i\|
    \leq R\varepsilon$,
by the property of \eqref{ineq:prop_rkhs} and that of the covering set.

With these results, we bound the following distance with the pair $f,g$ and any $x \in \mX$: 
\begin{align*}
    |f(x)-g(x)|&=|f(x)-f(x_i)+f(x_i)-g(x_i)+g(x_i)-g(x)|\\
    &\leq|f(x)-f(x_i)|+|f(x_i)-g(x_i)|+|g(x_i)-g(x)|\\
    &\leq (2R+1)\varepsilon,
\end{align*}
hence we have $\|f-g\|_{L^\infty} \leq (2R+1) \varepsilon$.

From the above discussion, the covering number $\mN((2R+1)\varepsilon,\mH_R,\|\cdot\|_{L^\infty})$ is bounded by the number of different $A_f$ when $f$ ranges over $\mH_R$. 
Since $|f(x)|\leq \|f\|_{\mathcal{H}_R}\leq R$ for any $x\in\mathcal{X}$ by \eqref{ineq:prop_rkhs}, the number of possible values of each element of $A_f$ is bounded by $({2R}/{\varepsilon}+1)$.
Assume the covering set $x_1,\dots,x_m$ is ordered such that $i<j$ implies $\|x_i - x_j\|<2\varepsilon$, then we obtain
$|f(x_j)-f(x_i)|\leq R \|x_j - x_i \| <2R\varepsilon$.
Therefore, we obtain 
\begin{align*}
    -2R\varepsilon+f(x_i)<f(x_j)<2R\varepsilon+f(x_i).
\end{align*}
It implies that for given $f(x_i)$ the number of possible values of $\lfloor {f(x_j)}/{\varepsilon}\rfloor$ is at most $4R+1$. 
Hence, we can bound the covering number as
\begin{align*}
    \mN((2R+1)\varepsilon,\mH_R,\|\cdot\|_{L^\infty})& = |\{A_f: f \in \mH_R\}| \\&\leq ({2R}/{\varepsilon}+1) (4R+1)^{m-1}\\& \leq ({2R}/{\varepsilon}+1) (4R+1)^{\exp{(c \varepsilon^{-\gamma})}-1}
\end{align*}
As a result, we obtain the following bound in the norm $\|\cdot\|_{L^\infty}$:
\begin{align*}
&\log N(\varepsilon, \mH_R,\|\cdot\|_{L^\infty})\\
&\leq \log\left\{\frac{2R(2R+1)}{\varepsilon}+1\right\}\,+ \left[\exp{\left\{
V \left(\frac{2R+1}{\varepsilon}\right)^{\gamma}\right\}}-1\right]\,\log(4R+1). 
\end{align*} 
Since the empirical norm $\|\cdot\|_n$ possesses the Riesz property (e.g. page 83 in \cite{van1996weak}), we obtain
\begin{align*}
    &\log N(\varepsilon,\mH_R,\|\cdot\|_{n}) \\
    &\leq \left({2R(2R+1)} / {\varepsilon}+1\right) + \left[\exp \left\{V \left(\frac{2R+1}{\varepsilon}\right)^{\gamma}\right\}-1\right] \log(4R+1) \\
    &\leq 6 \frac{R^2}{\varepsilon}\,+4R \left[\exp \left\{V \left(\frac{3R}{\varepsilon}\right)^{\gamma}\right\}-1\right],
\end{align*}
by the setting $R \geq 1$.
Then, we obtain the statement.
\end{proof}

The following result is to bound the Rademacher complexity by the covering number.
Although technique follows a standard discussion by the Dudley's integral, the covering number for functional data analysis has a specific role from functional data.
\begin{lemma}\label{lem:bound_rn}
Suppose Assumption \ref{asmp:covering} holds. 
Then, we obtain
\begin{eqnarray*}
    \mR_n(\mH_R) \leq {Rc_{V,\gamma}}{(\log n)^{-1/{\gamma}}}.
\end{eqnarray*}
\end{lemma}
\begin{proof}[Proof of Lemma \ref{lem:bound_rn}]

Now, we bound $\mR_n(\mH_R)$ using the following inequality learned from \cite{srebro2010note}:
\begin{align*}
    R_n(\mH_R) \leq \inf_{\alpha\geq 0}\left\{\,4\alpha +12\int_{\alpha}^{\sup_{f\in \mH_R}\sqrt{P_nf^2}} \sqrt{\frac{\log N(\varepsilon,\mH_R,\|\cdot\|_n)}{n}} \,d\varepsilon \,\right\} 
\end{align*}
As $\sup_{x\in \mX}|f(x)|\leq \|f\|_{\mH_R}\leq R$ for all $f\in \mH_R$, we obtain $\sup_{f\in \mH_R}\sqrt{P_nf^2} =R$. 
Lemma \ref{lem:covering} yields
\begin{align*}
    &\int_{\alpha}^R \sqrt{\frac{\log \mN(\varepsilon,\mH_R,\|\cdot\|_n) }{n}} \,d\varepsilon \\ &\leq\frac{1}{\sqrt{n}}\int_{\alpha}^R \sqrt{ \,\frac{6R^2}{\varepsilon}\,+4R \left[\exp \left\{V \left(\frac{3R}{\varepsilon}\right)^{\gamma}\right\}-1 \right] }d\varepsilon\\
    &\leq\frac{1}{\sqrt{n}} \int_{1}^{\frac{1}{\varepsilon_0}}  \sqrt{ 6R\tau\,+4R \,\{\exp{(V3^{\gamma} \tau^{\gamma})}-1 \,\} }\,\frac{R}{\tau^2}d\tau \\
    &\leq\frac{R}{\sqrt{n}}\int_{1}^{\frac{1}{\varepsilon_0}} \sqrt{6R}\tau^{-{3} / {2}}\,+2\sqrt{R}\tau^{-2} \{\exp{(V3^{\gamma}\tau^{\gamma})}-1\} d\tau\\
    &\leq\frac{R}{\sqrt{n}}\int_{1}^{\frac{1}{\varepsilon_0}} \sqrt{6R}\tau^{-{3} / {2}}\,+2\sqrt{R}\tau^{\gamma-1} \exp{(V3^{\gamma}\tau^{\gamma})}d\tau\\
    &\leq{R}{n^{-1/2}} \{2\sqrt{6R}+{2\sqrt{R}}\exp(V3^{\gamma}\varepsilon_0^{-\gamma}) / ({V3^{\gamma}})\}\\
    &\leq{R}{n^{-1/2}} \{ 2\sqrt{R}(\sqrt{6}+\frac{1}{V3^{\gamma}})\exp(V3^{\gamma})\}  \exp(\varepsilon_0^{-\gamma}).
\end{align*}
Here, we substitute $\tau =R/{\varepsilon}$ and define $\varepsilon_0=\alpha/{R}$. Then, we have
\begin{align*}
    \mR_n(\mH_R) &\leq \,R\inf_{0\leq \varepsilon_0 \leq 1}[4\varepsilon_0 +n^{-1/2} \{ 2\sqrt{R}(\sqrt{6}+\frac{1}{V3^{\gamma}})\exp(V3^{\gamma})\}  \exp(\varepsilon_0^{-\gamma})] \\
    &\leq R[4 {(\log n^{1/4})^{-1 /{\gamma}}} +{n^{-1/4}\{ 2\sqrt{R}(\sqrt{6}+ \frac{1}{V3^{\gamma}})\exp(V3^{\gamma})\}  } ]\\
    &\leq R\{ 2\sqrt{R}(\sqrt{6}+\frac{1}{V3^{\gamma}})\exp(V3^{\gamma})\} {(\log n)^{-1/{\gamma}}}.
\end{align*}
In the second inequality, we substitute $\varepsilon_0$ to $(\log n^{1/{4}})^{-1/{\gamma}}$.
Then, we obtain the statement.
\end{proof}

\section{Technical Lemma} \label{app:tech}
We provide several technical results for the proof of the main theorem.
\begin{lemma} \label{lem:func_derive}
    We obtain the following equality:
    \begin{align*}
        \nabla L_n(\hat{f}_n) = \frac{1}{n}\sum_{j=1}^n\ell'(Y_j\hat{f}_n(X_j))Y_jh(X_j)+2\lambda\langle\hat{f}_n,h\rangle_\mH.
    \end{align*}
\end{lemma}
\begin{proof}[Proof of Lemma \ref{lem:func_derive}]
We study the optimization problem in \eqref{def:erm} by considering its functional derivative in the Fr\'echet sense.
For a coefficient $\alpha > 0$, we rewrite the target function in \eqref{def:erm} as
\begin{align*}
    L_n(\alpha)&=P_n\{\ell\circ (\hat{f}_n+\alpha h)\}+\lambda\|\hat{f}_n+\alpha h\|_\mH^2=\frac{1}{n}\sum_{j=1}^n\ell\{Y_j(\hat{f}_n+\alpha h)(X_j)\}+\lambda\|\hat{f}_n+\alpha h\|_\mH^2.
\end{align*} 
Since $\hat{f}_n$ is the minimizer of the problem in \eqref{def:erm}, a derivative of $L_n(\alpha)$ is $0$ with $f = \hat{f}_n$ and $\alpha = 0$.
By the differentiability of $\ell$, we obtain the following derivative
\begin{eqnarray*}
    \frac{dL_n}{d\alpha}(0)=\frac{1}{n}\sum_{j=1}^n\ell'\{Y_j\hat{f}_n(X_j)\}Y_jh(X_j)+2\lambda\langle\hat{f}_n,h\rangle_\mH,
\end{eqnarray*}
then we obtain the statement.
\end{proof}

\begin{lemma} \label{lem:bound_xi}
Suppose $\g(x)\leq 0 $ and $ \|\g\|_\mH<U$ hold.
Then, for $x \in \mX$ such that $f_0(x) = \delta > 0$ holds, we set $S=B(x;\delta_0)$ obtain
\begin{align*}\label{bound_xi}
    \frac{1}{n}\sum_{j=1}^n\ell'\{Y_j\g(X_j)\}Y_jh(X_j) \leq \frac{1}{n}\sum_{j=1}^n\xi_j,
\end{align*}
 where $\xi_j:=2U\delta_0 h(X)I_{S}(X_j)+\ell'(0)Yh(X)I_{S}(X_j)+|\ell'(-U)| h(X)I_{S^c}(X_j)$.
\end{lemma}
\begin{proof}[Proof of Lemma \ref{lem:bound_xi}]
We prepare some inequalities.
It should be noted that
\begin{align*}
    &\frac{1}{n}\sum_{j=1}^n\ell'\{Y_j\g(X_j)\}Y_jh(X_j)\\
    &=\frac{1}{n}\sum_{j:Y_j=+1}^n\ell'\{\g (X_j)\}h(X_j)-\frac{1}{n}\sum_{j:Y_j=-1}^n\ell'\{-\g(X_j)\}h(X_j).
\end{align*}
Also, we recall that $\ell'$ is negative and increasing, $h$ is nonnegative. 
For any $x'\in S$, the RKHS property \eqref{ineq:prop_rkhs} provides
    $\g(x')\leq \|\g\|_\mH \delta_0\leq U\delta_0$,
then we have
    $\ell'\{\g(x')\} \leq \ell'(U\delta_0), \mbox{~and~} \ell'\{-\g(x')\} \geq \ell'(-U\delta_0)$.
Also, $|\g(x')|\leq \|\g\|_\mH\leq U$ holds for all $x'\in S$ suggests that
     $|\ell'\{\g(y)\}| \leq |\ell'(-U)|, \mbox{~and~} |\ell'\{-\g(y)\}| \leq |\ell'(-U)|$.

Now, we are ready to bound the target value.
For $j = 1,...,n$, we define $Z_j := h(X_j) I_S(X_j)$ and $Z_j^c := h(X_j) I_{S^c}(X_j)$ for brevity.
We bound the value as
\begin{align*}
  &\frac{1}{n}\sum_{j=1}^n\ell'\{Y_j\g(X_j)\}Y_jh(X_j)\\
  &\leq \frac{\ell'(U\delta_0)}{n}\sum_{j:X_j\in S,Y_j=+1} h(X_j)  -\frac{\ell'(-U\delta_0)}{n}\sum_{j:X_j\in S,Y_j=-1} h(X_j) +\frac{|\ell'(-U)|}{n}\sum_{j:X_j\in S^c} h(X_j)\\
  &= \frac{\ell'(U\delta_0)}{n}\sum_{j:X_j\in S} \frac{1+Y_j}{2}h(X_j)  -\frac{\ell'(-U\delta_0)}{n}\sum_{j:X_j\in S} \frac{1-Y_j}{2}h(X_j) +\frac{|\ell'(-U)|}{n}\sum_{j:X_j\in S^c} h(X_j)\\
  &= \frac{\ell'(U\delta_0)-\ell'(-U\delta_0)}{2n}\sum_{j=1}^nZ_j+ \frac{\ell'(U\delta_0)+\ell'(-U\delta_0)}{2n}\sum_{j=1}^nY_j Z_j  +\frac{|\ell'(-U)|}{n}\sum_{j=1}^n Z_j^c.
\end{align*}
About the coefficient terms, we obtain
\begin{align*}
    \left|\frac{\ell'(U\delta_0)+\ell'(-U\delta_0)}{2}-\ell'(0)\right|&\leq  \frac{|\ell'(U\delta_0)-\ell'(0)|}{2}+\frac{|\ell'(-U\delta_0)-\ell'(0)|}{2} \\
    &\leq \ell''(0)U \delta_0 \leq U\delta_0,
\end{align*}
and similarly
\begin{eqnarray*}
    \frac{|\ell'(U\delta_0)-\ell'(-U\delta_0)|}{2}\leq  \frac{|\ell'(U\delta_0)-\ell'(0)|}{2}+\frac{|\ell'(-U\delta_0)-\ell'(0)|}{2} \leq U\delta_0.
\end{eqnarray*}
Using the inequalities, we further bound the target value as
\begin{align*}
    &\frac{1}{n}\sum_{j=1}^n\ell'\{Y_j\g(X_j)\}Y_jh(X_j) \\
    &\leq \frac{U\delta_0}{n}  \sum_{j=1}^n Z_j+\frac{U\delta_0+\ell'(0)}{n}\sum_{j=1}^nY_j Z_j+\frac{|\ell'(-U)|}{n}\sum_{j=1}^n Z_j^c  \\
    &\leq 2\frac{U\delta_0}{n}  \sum_{j=1}^n Z_j+\frac{\ell'(0)}{n}\sum_{j=1}^nY_j Z_j+\frac{|\ell'(-U)|}{n}\sum_{j=1}^n Z_j^c.
\end{align*}
Then, we obtain the statement by the definition of $\xi_j$.
\end{proof}

\begin{lemma}\label{lem:bern_bound} Consider the same setting with Lemma \ref{lem:bound_xi}.
Then, we get the following inequality:
\begin{eqnarray*}
    \Pr\left(\frac{1}{n}\sum_{j=1}^n\xi_j\geq -\frac{1}{2}\delta_0\Ep[h(X)]\right)
    \leq2\exp \left\{-\frac{n\delta_0\Ep[h(X)]}{C_{U,L}}\right\}.
\end{eqnarray*}
\end{lemma}
\begin{proof}[Proof of Lemma \ref{lem:bern_bound}]
We firstly bound an expectation of $\xi_j$, 
\begin{align*}
    \Ep[\xi_j]&=2U\delta_0\Ep[h(X)I_{S}(X)]+\ell'(0)\Ep[Yh(X)I_{S}(X)]+|\ell'(-U)|\Ep[h(X)I_{S^c}(X)] \\
    & \leq 2U\delta_0\Ep[h(X)]+\ell'(0)\Ep[Yh(X)I_{S}(X)]+|\ell'(-U)| \delta \Ep[h(X)],
\end{align*}
by the conditions of $\mH(x,\delta)$ presented in \eqref{cond:rkhs}.
We define $\bar{f}(x) =\tilde{f}^*(x) \vee 1$.
For the term $\Ep[Yh(X)I_{S}(X)]$, we approximate it as
\begin{align*}
    \Ep[Yh(X)I_{S}(X)] &\geq \Ep[\bar{f}(X)h(X)I_{S}(X)] \\
    &\geq (1 - L \delta_0) \Ep[h(X)I_{S}(X)]\\
    &\geq (1 - L \delta_0)(1-\delta_0)\Ep[h(X)].
\end{align*}
The first equality holds since $\tilde{f}^*$ is a perfect classifier.
The second equality follows the hard-margin condition on $\tilde{f}^*$ and  $\inf_{x' \in S} \tilde{f}^*(x') \geq \delta - L \delta_0$ by the Lipschitz constant $L$ of $\tilde{f}^*$.
For the last inequality, we apply the condition \textit{(iii)} for $\mH(x,\delta_0)$ in \eqref{cond:rkhs} and obtain
\begin{align*}
    \delta_0\int_\mX hd\Pi\geq \int_{S^c}hd\Pi =\Ep[h]-\int_{S}hd\Pi, 
\end{align*}
then we have $\Ep[h(X)I_{S}(X)]\geq (1-\delta_0)\Ep[h(X)]$.
Since  $\ell'(0)<0$ holds, we substitute the bound for $\Ep[Yh(X)I_{S}(X)]$ and obtain
\begin{align*}
    \Ep[\xi_j]\leq \{ 2U\delta_0+  \ell'(0) (1-L\delta_0)(1-\delta_0)+|\ell'(-U)|\delta_0 \}\Ep[h(X)] \leq -\delta_0 \Ep[h(X)].
\end{align*}
The last inequality follows by selecting a sufficiently small $\delta_0 > 0$ as $\delta_0 \leq 1/({L+4U+12})$.

We finally bound a tail probability of $n^{-1}\sum_{j=1}^n\xi_j$ by the Bernstein inequality (Theorem 3.1.7 in \cite{gine2016mathematical}).
Using an elementary inequality $(a+b+c)^2\leq 3a^2+3b^2+3c^2$, we have
$\Ep(\xi^2) \leq C\delta_0 \Ep(h(X))$.
In addition, it is clear that $|\xi|\leq C_{U,L}\delta_0$.
Then, by the Bernstein's inequality, we obtain
\begin{align*}
    \Pr \left(\frac{1}{n}\sum_{j=1}^n\xi_j\geq -\frac{1}{2}\delta_0\Ep[h(X)]\right)&\leq 2 \exp\left\{-\frac{n^2\delta_0^2\Ep[h(X)]^2/8}{\sum_{i=1}^n\Ep[\xi_i^2]-\,\,nC_{U,L}\delta_0^2\Ep[h(X)]/6}\right\} \\
    &\leq 2\exp \left\{-\frac{n\delta_0\Ep[h(X)]}{C_{U,L}}\right\}.
\end{align*}
Then, we obtain the statement.
\end{proof}

\section{Additional Experiment: Different Eigenfunctions}
We implement an additional experiment to validate the main result when the covariance functions between different labels do not have the same eigenfunctions.
The setting in this section does not strictly satisfy our assumptions, therefore it is outside the scope of our theory. 
However, to investigate the potential applicability of our theory, we perform this experiment with different hyper-parameter choices.

\subsection{Experimental Setting}

As in Section 4, we generate functional data from two groups with labels $\{-1,1\}$.
For each group, we generate $n$ functions on $\mT = [0,1]$ with two orthogonal bases:\begin{align*}
  \left\{
  \begin{array}{ll} 
  \phi_0(t) = 1, \ \ \ \ \phi_j(t) = \sqrt{2} \sin (\pi jt),\ \ \  \forall j \geq 1,  \\
  \psi_0(t) = 1 , \ \ \ \ \psi_j(t) = \sqrt{2} \cos (\pi jt),\ \ \  \forall j \geq 1.
  \end{array}
  \right.
\end{align*}
$n$ is set from $1$ to $3000$.
For a label $+1$, we generate functional data $X_{i+}(t) = \sum_{j=0}^{50} (\theta_j^{1/2} Z_{j+} + \mu_{j+}) \phi_j(t)$ with random variables $Z_{j+}$ and coefficients $ \theta_j, \mu_{j+}$ for $j=0,1,...,50$ and $i=1,...,n$.
For a label $-1$, we generate $X_{i-}(t) = \sum_{j=0}^{50} (\theta_j^{1/2} Z_{j-} + \mu_{j-}) \psi_j(t)$ with random variables $Z_{j-}$ and coefficients $ \mu_{j-}$.

That is, in \textit{Scenario 1}, we set $\theta_j = j^{-2}$, $\mu_{j-}=0$, and change $\mu_{j+} = j^{-\gamma}$ and draw $Z_{j+},Z_{j-}$ from standard normal Gaussian. Then we determined the DH condition based on whether the gamma was greater or less than $3/2$.
In \textit{Scenario 2}, we set $\theta_j = j^{-2},$ $\mu_{j-}= 0$ and adjust $\mu_{j+} = \mone\{j = 0\} \mu$, and let $Z_{j+},Z_{j-}$ be sample from uniform distribution on $[-1/2,1/2]$. Although it is analytically challenging to specify when the HM condition is violated because of the different basis functions, we present the results of our experiments for various $\mu$.
In Scenario 2, because of the difficulty of rigorously checking the DH condition in the setting, we examined a broader range of $\mu \in \{1.5, 1.7, 1.9, 2.1\}$.

Other settings are the same as Section 4.
For each $n$, we newly generate $1000$ test data and calculate misclassification rates with the test data. 
Each simulation experiment is repeated $200$ times, and the average value is reported.
We investigate the classification error of the RKHS classifier with the Gaussian kernel and the logit loss. The tuning parameters are chosen by cross-validation.

\subsection{Result}

\begin{figure}[tbp]
\centering
  \hspace*{-1cm}
  \includegraphics[width=\hsize]{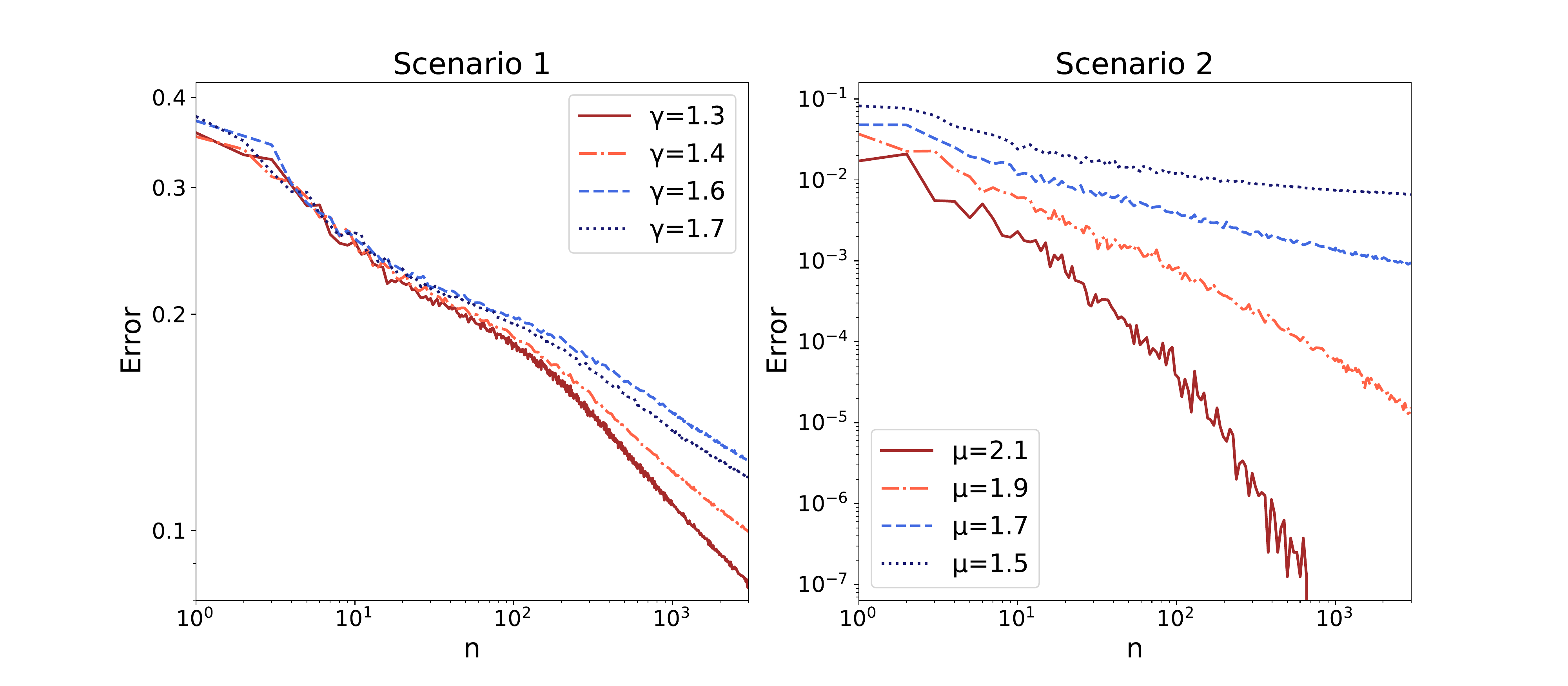}
  \caption{Error (logarithm of misclassification error rate) by the RKHS against $\log n$.
  Left: Scenario 1 for the Delaigle--Hall condition with $\gamma \in \{1.3\,\mathrm{(solid)},1.4\,\mathrm{(dashes)},1.6\,\mathrm{(dots)},1.7\,\mathrm{(dotdash)}\}$. Right: Scenario 2 for the hard-margin condition with $\mu \in \{2.1\,\mathrm{(solid)},1.9\,\mathrm{(dashes)},1.7\,\mathrm{(dot)},1.5\,\mathrm{(dotdash)}\}$.
\label{fig:add}}
\end{figure}

The results are shown in Figure \ref{fig:add}. 
In Scenario 1, we see a difference in convergence speed for each value of $\gamma$, although the difference is not so clear as in Figure 1 in the main text.
The reason is that the conditions we are actually checking are different from those we should impose, so there must be a difference in scale. 
However, as $n$ increases to some extent, e.g. $n\ge 100$, we can observe an exponential-like fast convergence.
In Scenario 2, 
the results confirm exponential convergence for large $\mu$, and as $\mu$ gets smaller, the rate falls off as in polynomial convergence.

\end{document}